\documentclass[reqno]{amsart}

\usepackage{amsmath,amsthm,amssymb}
\usepackage{mathtools}
\usepackage{tikz}
\usetikzlibrary{calc,shapes.geometric,matrix}
\usepackage{multicol}
\usepackage{array}
\usepackage[colorlinks=false]{hyperref}
\usepackage{cite}
\usepackage{microtype}
\usepackage{paracol}  
\usepackage{graphicx}
\usepackage{thm-restate}

\title{CAT(0) and cubulated Shephard groups}
\author{Katherine M. Goldman}

\newtheorem{thm}{Theorem}[section]
\newtheorem{prop}[thm]{Proposition}
\newtheorem{lemma}[thm]{Lemma}
\newtheorem{cor}[thm]{Corollary}

\theoremstyle{definition}
\newtheorem{defn}[thm]{Definition}
\newtheorem{ex}[thm]{Example}
\newtheorem*{remark}{Remark}

\newtheorem*{cor*}{Corollary}

\numberwithin{equation}{section}

\newcommand{\tempname}{CCCC}
\newcommand{\bigast}{\mathop{\scalebox{1.8}{\raisebox{-0.2ex}{$\ast$}}}}

\begin{document}

\begin{abstract}
    Shephard groups are common generalizations of Coxeter groups, Artin groups, and graph products of cyclic groups. Their definition is similar to that of a Coxeter group, but generators may have arbitrary order rather than strictly order 2. 
    We extend a well known result that Coxeter groups are $\mathrm{CAT}(0)$ to a class of Shephard groups that have ``enough'' finite parabolic subgroups. 
    We also show that in this setting, if the associated Coxeter group is type (FC), then the Shephard group acts properly and cocompactly on a $\mathrm{CAT}(0)$ cube complex.
    As part of our proof of the former result, we introduce a new criteria for a complex made of $A_3$ simplices to be $\mathrm{CAT}(1)$.
\end{abstract}

\maketitle
    
\section{Introduction}

The classical notion of a \emph{Shephard group} first arose in G.~C.~Shephard's thesis \cite{shep52}. In it, Shephard introduces the notion of a regular complex polytope and studies the symmetry groups of such an object, which have been termed Shephard groups after his work. It turns out that these symmetry groups are complex reflection groups (finite subgroups of $GL_n(\mathbb{C})$ generated by complex linear reflections) and have a ``Coxeter-like'' presentation. This group presentation can be encoded in a ``Coxeter-like'' diagram, with certain restrictions on labels and shape of the diagram. In this paper, we study those groups defined by presentations identical to those of the classical Shephard groups, but without these same restrictions on the diagrams. Their precise definition is as follows.

Let $\Gamma$ be a simplicial graph with vertex set $I$ and the following information. 
For each vertex $i \in I$, we assign a number $p_i \in \mathbb{Z}_{\geq 2} \cup \{\infty\}$. 
For each edge $\{i,j\}$ of $\Gamma$, we assign a number $m_{ij} = m_{ji} \in \mathbb{Z}_{\geq 3} \cup \{\infty\}$. 
If $\{i,j\}$ is not an edge, define $m_{ij} = 2$. 
If $m_{ij}$ is odd, then we require that $p_i = p_j$. 
Sometimes if $p_i = 2$ or $m_{ij} = 3$ we omit the label.
We say this data defines an \emph{extended Coxeter diagram}. 
Associated to such a diagram $\Gamma$, we define the \emph{Shephard group}
$G_\Gamma$ on generators $S = \{\,s_i:i \in I\,\}$ with the following presentation:
\begin{equation} \label{def:sheppres}
    G_\Gamma = \left\langle
    S \ \middle| \ \begin{matrix}
        \mathrm{prod}(s_i,s_j; m_{ij}) = \mathrm{prod}(s_j,s_i; m_{ij}) \\
        s_i^{p_i} = 1
    \end{matrix}
     \right\rangle,
\end{equation}
where $\mathrm{prod}(a,b; m)$ denotes $(ab)^{m/2}$ if $m$ is even and $(ab)^{(m-1)/2}a$ if $m$ is odd.
(When $p_i = \infty$ or $m_{ij} = \infty$, the corresponding relation is omitted.) This is why we require $p_i = p_j$ when $m_{ij}$ is odd; $s_i$ and $s_j$ are conjugate, and thus if $p_i \neq p_j$ the order of $s_i$ and $s_j$ would not necessarily be $p_i$ or $p_j$, respectively. Note that if $\Gamma$ is disconnected, then $G_\Gamma$ splits as a direct product over the connected components of $\Gamma$. 

We point out some special cases to show the power of this general definition.
\begin{enumerate}
    \item If all $p_i = 2$, then $G_\Gamma$ is a Coxeter group.
    \item If all $p_i = \infty$, then $G_\Gamma$ is an Artin group.
    \item If all $m_{ij} = 2$ or $\infty$ (in which case we may call $G_\Gamma$ ``right-angled''), then $G_\Gamma$ is a graph product of cyclic groups.
\end{enumerate}
This is, however, not simply an empty generalization to compile the above groups into one succinct class. Many interesting infinite Shephard groups (and their quotients) appear in various places in mathematics. For example, they arise naturally in algebraic geometry as objects closely related to moduli spaces of arrangements in $\mathbb{CP}^2$ \cite{kapovich1998representation}. 
We detail further examples in the following.

\begin{ex}
In certain cases, Shephard groups have (typically proper) quotients to infinite affine and hyperbolic complex reflection groups; a particularly interesting class of examples comes from the Shephard groups $G_\Gamma$ with diagram $\Gamma$ given in Figure \ref{fig:infinitetriangles}.

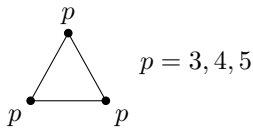
\begin{figure}[ht]
    \begin{tikzpicture}
        \filldraw (0,0) circle (0.05cm);
        \filldraw (1,0) circle (0.05cm);
        \filldraw (0.5,0.9) circle (0.05cm);
        \draw (0,0) -- (1,0) -- (0.5,0.9) -- (0,0);
        \node[above] at (0.5,0.9) {$p$};
        \node[below left] at (0,0) {$p$};
        \node[below right] at (1,0) {$p$};
        \node at (2.2,0.5) {$p = 3,4,5$};
    \end{tikzpicture}
    \caption{A class of infinite Shephard groups}
    \label{fig:infinitetriangles}
\end{figure}

When $p = 3$, $G_\Gamma$ has a quotient to one of Popov's \emph{affine complex reflection groups} \cite{popov22complex}. It is demonstrated in \cite{cote2021complex} that this quotient (refered to as \textsc{Refl}$(\widetilde{G_4})$ in said article) possesses a quite elegant geometry: the complement of its reflection hyperplanes deformation retracts to a non-positively curved 2-dimensional complex built out of $(3,3,3)$ triangles, whose link is a Mobius-Kantor graph. 
This link is isomorphic (and isometric) to the link of a complex we define for arbitrary Shephard groups---but this $(3,3,3)$ triangle complex for \textsc{Refl}$(\widetilde{G_4})$ is \emph{not} the same as our complex for $G_\Gamma$. (It may be interesting to further study the exact relationship between these two complexes.)

Further, when $p = 3,4$ or $5$, $G_\Gamma$ has quotients to the hyperbolic complex reflection groups upon which Mostow's famous examples of complex hyperbolic polyhedra in $\mathbb{CH}^2$ and non-arithmetic lattices of $\mathrm{PU}(2,1)$ are based \cite{mostow1980remarkable}.
\end{ex}

Another interesting pair of examples comes from the following diagram, which we call $B_n(p,q)$:
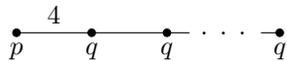
\[
\begin{tikzpicture}
    \filldraw (0,0) circle (0.05cm);
    \draw (0,0) -- node[above] {$4$} (1,0);
    \draw (1,0) -- (2.3,0);
    \filldraw (1,0) circle (0.05cm);
    \filldraw (2,0) circle (0.05cm);
    \node at (2.75,0) {. . .};
    \filldraw(3.5,0) circle (0.05cm);
    \draw (3.5,0) -- (3.25,0);
    \node at (0,-0.25) { $p$ };
    \node at (1,-0.25) { $q$ };
    \node at (2,-0.25) { $q$ };
    \node at (3.5,-0.25) { $q$ };
\end{tikzpicture}
\]
This family of Shephard groups finds use in the dual Garside theory of Artin groups. The dual Garside theory has shown to be quite effective in making progress on the $K(\pi,1)$ conjecture (e.g., in the recent proof of the conjecture for the affine-type Artin groups \cite{paolini2021proof}). 

\begin{ex} 
In \cite{baumeister2017simple}, it is shown that the Artin group of type $D_n$ embeds into the Shephard group of type $B_n(2,\infty)$ as a finite-index subgroup. 
This is used to classify the so-called \emph{Mikado braids} (certain well-behaved words) of the $D_n$ Artin group, and specifically, to show that the set of simple elements under its dual Garside structure are all Mikado braids. 
\end{ex}

\begin{ex}
In \cite{mccammond2017artin}, the Shephard group of type $B_n(\infty,2)$ is called the ``middle group'' \textsc{Mid}$(B_n)$, and is defined explicitly as a certain semidirect product of translations and reflections in $\mathbb{R}^n$. This group plays a key technical role in the proofs of the facts that, for a general affine-type Artin group, it is isomorphic to its dual Artin group, and this dual Artin group embeds in a Garside group. (It is these results that form part of the foundation of \cite{paolini2021proof}.)
\end{ex}

Much of our work relies on the close relationship of Shephard groups and Coxeter groups, and so we make the following definition. For any extended Coxeter diagram, we let $W_\Gamma$ and $A_\Gamma$ denote the Coxeter group and Artin group, resp., defined by the underlying Coxeter diagram of $\Gamma$ (that is, the Coxeter diagram obtained from $\Gamma$ by ignoring the vertex labels).
We recall that these groups have the following presentations:
\begin{align*}
    A_\Gamma &= \left\langle\
    S \ \middle| \ 
        \mathrm{prod}(s_i,s_j; m_{ij}) = \mathrm{prod}(s_j,s_i; m_{ij})\
     \right\rangle, \\
    W_\Gamma &= \left\langle\
    S \ \middle| \
        (s_is_j)^{m_{ij}} = 1,
        s_i^{2} = 1 \
     \right\rangle,
\end{align*}
If the diagram $\Gamma$ has all $p_i = 2$ (so $G_\Gamma = W_\Gamma$), we will sometimes call $\Gamma$ ``Coxeter'' (without the ``extended'') and otherwise call $\Gamma$ ``non-Coxeter''. But we emphasize that we declare $W_\Gamma$ is a Coxeter group regardless of the vertex labeling of $\Gamma$.

\subsection{Description of results}

We now give a broad overview of the results of this paper. 

We heavily use the fact that $G_\Gamma$ is finite for certain diagrams $\Gamma$.
If $G_\Gamma$ is a finite group, we sometimes call the diagram $\Gamma$ itself ``finite''.
It is an interesting fact that there are finite (abstract) Shephard groups $G_\Gamma$ which are not themselves Coxeter groups. 
See Table \ref{tab:finiteshephards} for examples.
The diagrams in this table come from the the well-known classification of finite complex reflection groups. 
However, it is a priori unclear if there is some diagram $\Gamma$, say, which is branched or a cycle, so that $G_\Gamma$ is a finite group under this abstract presentation, but does not correspond to any complex reflection group.
We confirm in Theorem \ref{thm:classifyfinite} that this cannot happen; that is, the only finite abstract Shephard groups $G_\Gamma$ are those arising from the complex reflection groups (specifically, those in Tables \ref{tab:finitediagrams} and \ref{tab:finiteshephards}).

We are able to classify the finite abstract Shephard groups, but the infinite Shephard groups remain rather mysterious. 
It is unclear if there is a unified topological interpretation of the Shephard groups, compared to the relationship between Coxeter groups and real reflection groups, and that between Artin groups and hyperplane complements. One of the main objectives of this article is to propose a geometric model for the infinite Shephard groups which, in many cases, can play the role of the Davis complex for a given Coxeter group. Similarly to the Davis complex, this model is built out of the finite ``parabolic'' subgroups of the Shephard group. To make this more precise, we make the following definitions.

For an extended Coxeter diagram $\Gamma$, 
a \emph{subdiagram} is a \textbf{full} subgraph $\Gamma'$ of $\Gamma$ inheriting the vertex and edge labels of $\Gamma$. (Recall that a subgraph is full if whenever two vertices of the subgraph are joined by an edge in the original graph, they remain joined by an edge in the subgraph.)
If $\Gamma'$ is a subdiagram of $\Gamma$, then $G_{\Gamma'}$ is also a Shephard group. 
The group $G_{\Gamma'}$ surjects onto the subgroup of $G_{\Gamma}$ generated by the vertex set of $\Gamma'$ viewed within $\Gamma$. To state this more precisely, and momentarily avoid the question of whether this is an isomorphism or not, we introduce the following notation: let $S = \{\,s_i:i \in I\,\}$ denote the generating set of $G_\Gamma$, and let $J \subseteq I$ with $T = \{\,s_j:j \in J\,\} \subseteq S$. Then define $G_T = \langle T \rangle$ to be the subgroup of $G_\Gamma$ generated by $T$, and define $G_{\Gamma(T)}$ to be the Shephard group defined by $\Gamma(T)$, the full subgraph of $\Gamma$ on vertex set $J$.

We now describe the main complex on which the Shephard groups act. Consider the following posets.
\begin{align*}
    \mathcal S^f_\Gamma    = \mathcal S^f    &= \{\,T \subseteq S : W_{\Gamma(T)} \text{ is finite}\,\}, \\
    \mathcal S^{fs}_\Gamma = \mathcal S^{fs} &= \{\,T \subseteq S : G_{\Gamma(T)} \text{ is finite}\,\}.
\end{align*}
The set $\mathcal S^f$ plays a major role in the study of Coxeter groups as the basis of the fundamental domain of the Davis complex. Its direct analogue in a Shephard group is $\mathcal S^{fs}$. We continue the analogy with Coxeter groups to define a complex (in Definition \ref{def:complexofshephard}), which we denote $\Theta = \Theta_\Gamma$, on which $G_\Gamma$ acts properly and cocompactly.

It would be too much to hope that this complex behaves nicely for all Shephard groups; after all, $\Gamma$ could have all vertex labels $\infty$, in which case $\mathcal S^{fs}$ would be empty. But with certain restrictions on the number and type of finite subgroups, we can show that $\Theta$ has a very nice geometry, and is, in fact, $\mathrm{CAT}(0)$. Our first step in this direction is the following.

\begin{restatable}{theorem}{moussongcat0}
\label{thm:moussong}
    Suppose $\Gamma$ is an extended Coxeter diagram with $\mathcal S^{f}_\Gamma = \mathcal S^{fs}_\Gamma$ such that $\Gamma$ has no subdiagram of type ``$A_4(3)$'' (see Table \ref{tab:finiteshephards}). Then $\Theta_\Gamma$ is $\mathrm{CAT}(0)$, and hence $G_\Gamma$ is $\mathrm{CAT}(0)$ (i.e., acts properly and cocompactly on a $\mathrm{CAT}(0)$ space by isometries).
\end{restatable}

Some examples of extended Coxeter diagrams satisfying the hypothesis of Theorem \ref{thm:moussong} can be found in Table \ref{tab:excat0}.
Among the other nice properties of $\mathrm{CAT}(0)$ groups, notable consequences of Theorem \ref{thm:moussong} for the applicable $G_\Gamma$ are:

\begin{cor*}
    Suppose $\Gamma$ is an extended Coxeter diagram with $\mathcal S^{f}_\Gamma = \mathcal S^{fs}_\Gamma$ which has no subdiagram of type $A_4(3)$. Then the word problem and congugacy problem for $G_\Gamma$ are solvable, and $G_\Gamma$ has quadratic Dehn function.
\end{cor*}

Along the way, we extract the following combinatorial criteria of Charney introduced in \cite{charney2004deligne} (based on \cite{elder2002curvature}) to show that a specific 2-dimensional piecewise spherical complex is $\mathrm{CAT}(1)$. See Definition \ref{def:cccc} and Theorem \ref{thm:charneycat1} for more detailed descriptions.

\begin{restatable}{theorem}{firstCharney}
\label{thm:firstCharney}
    Suppose $\Psi$ is a marked $A_3$ simplicial complex.
    If $\Psi$ is locally $\mathrm{CAT}(1)$ with the link of the $\pi/2$ vertices complete bipartite, and if the 4-cycles and 6-cycles in the 1-skeleton of $\Psi$ seen in Figure \ref{fig:shortclosededge} can be ``filled in'' to the subcomplexes in Figure \ref{fig:fillededge}, then $\Psi$ is $\mathrm{CAT}(1)$. 
\end{restatable}

Our second result is comparable to \cite[Thm.~4.3.5]{charney1995k}. There are certain Shephard groups where $\Theta_\Gamma$ has a natural cubical structure:

\begin{restatable}{theorem}{cubecat0}
\label{thm:cube}
    Suppose $\Gamma$ is an extended Coxeter diagram so that the underlying Coxeter diagram is type \emph{(FC)} and $\mathcal S^f_\Gamma = \mathcal S^{fs}_\Gamma$. Then $\Theta_\Gamma$ is a $\mathrm{CAT}(0)$ cube complex, and hence $G_\Gamma$ is (cocompactly) cubulated.
\end{restatable}

\begin{table}[p]
    \centering

    \vspace{2em}

    \begin{tikzpicture}
        \filldraw (0,0) circle (0.05cm);
        \filldraw (1,0) circle (0.05cm);
        \filldraw (2,0) circle (0.05cm);
        \filldraw (4,0) circle (0.05cm);
        \filldraw (5,0) circle (0.05cm);
        \filldraw (6,0) circle (0.05cm);
        \draw (0,0) -- node[above] {$4$} (1,0);
        \draw (1,0) -- (5,0);
        \draw (5,0) -- node[above] {$4$} (6,0);
        \node[fill=white,circle] (dots) at (3,0) {$\cdots$};
        \node[below=0.1cm] at (0,0) {$p$};
        \node[below=0.1cm] at (6,0) {$q$};
        \node at (7.5,0) {$p,q \in \mathbb{Z}_{\geq 2}$};
    \end{tikzpicture}
    \vspace{1em}

    \begin{tikzpicture}
        \filldraw (0,0) circle (0.05cm);
        \filldraw (1,0) circle (0.05cm);
        \filldraw (2,0) circle (0.05cm);
        \filldraw (4,0) circle (0.05cm);
        \filldraw (5,0) circle (0.05cm);
        \filldraw (6,0.5) circle (0.05cm);
        \filldraw (6,-0.5) circle (0.05cm);
        \draw (0,0) -- node[above] {$4$} (1,0);
        \draw (1,0) -- (5,0);
        \draw (5,0) -- (6,0.5);
        \draw (5,0) -- (6,-0.5);
        \node[fill=white,circle] (dots) at (3,0) {$\cdots$};
        \node[below=0.1cm] at (0,0) {$p$};
        \node at (7.5,0) {$p \in \mathbb{Z}_{\geq 2}$};
    \end{tikzpicture}
    \vspace{1em}

    \begin{tikzpicture}
        \filldraw (0,0) circle (0.05cm);
        \filldraw (1,0) circle (0.05cm);
        \filldraw (-1,0) circle (0.05cm);
        \draw (0,0) --node[above] {$4$} (1,0);
        \draw (0,0) --node[above] {$4$} (-1,0);
        \node[below] at (0,0) {$p$};
        \filldraw ( 2,0) circle (0.05cm);
        \filldraw (-2,0) circle (0.05cm);
        \draw (2,0) -- (1,0);
        \draw (-2,0) -- (-1,0);
        \filldraw ( 3.5,0) circle (0.05cm);
        \filldraw (-3.5,0) circle (0.05cm);
        \draw (2,0) -- (3.5,0);
        \draw (-2,0) -- (-3.5,0);
        \node[fill=white,circle] (dots) at (2.75,0) {$\cdots$};
        \node[fill=white,circle] (dots) at (-2.75,0) {$\cdots$};
        \node at (4.8,0) {$p \in \mathbb{Z}_{\geq 2}$};
    \end{tikzpicture}
    \begin{tikzpicture}
        \matrix[column sep=20pt,row sep=6pt, nodes={anchor=center, minimum height=0.1cm, align=flush center}]{
        \filldraw (0,0) circle (0.05cm);
        \filldraw (1,0) circle (0.05cm);
        \filldraw (0.5,0.9) circle (0.05cm);
        \draw (0,0) -- (1,0) -- (0.5,0.9) -- (0,0);
        \node[above] at (0.5,0.9) {$p$};
        \node[below left] at (0,0) {$p$};
        \node[below right] at (1,0) {$p$};
        \node at (2.2,0.5) {$p = 2,3,4,5$};
        & 
        \filldraw (0,0) circle (0.05cm);
        \filldraw (1,0) circle (0.05cm);
        \filldraw (0.5,0.9) circle (0.05cm);
        \draw (0,0) --node[below] {$4$} (1,0) -- (0.5,0.9) -- (0,0);
        \node[above] at (0.5,0.9) {$p$};
        \node[below left] at (0,0) {$p$};
        \node[below right] at (1,0) {$p$};
        \node at (2.2,0.5) {$p = 2,3,4$};
        &
        \filldraw (0,0) circle (0.05cm);
        \filldraw (1,0) circle (0.05cm);
        \filldraw (0.5,0.9) circle (0.05cm);
        \draw (0,0) -- (1,0) --node[right] {$4$} (0.5,0.9) --node[left] {$4$} (0,0);
        \node[above] at (0.5,0.9) {$p$};

        \node at (2.1,0.5) {$p \in \mathbb{Z}_{\geq 2}$};
        \\
        };
    \end{tikzpicture}
    \begin{tikzpicture}
        \matrix[ column sep=25pt,row sep=6pt, nodes={anchor=center, minimum height=0.1cm, align=flush center}]{
        \filldraw (0,0) circle (0.05cm);
        \filldraw (1,0) circle (0.05cm);
        \filldraw (2,0) circle (0.05cm);
        \filldraw (3,0) circle (0.05cm);
        \draw (0,0) -- node[above] {$4$} (1,0);
        \draw (1,0) -- (2,0);
        \draw (2,0) -- node[above] {$4$} (3,0);
        
        \node[below=0.1cm] at (1,0) {$3$};  
        \node[below=0.1cm] at (2,0) {$3$};
        
        &
        
        \filldraw (1,0) circle (0.05cm);
        \filldraw (2,0) circle (0.05cm);
        \filldraw (3,0) circle (0.05cm);
        
        \draw (1,0) --node[above] {$5$} (2,0);
        \draw (2,0) --node[above] {$4$}  (3,0);
        
        \node[below] at (1,0) {$3$};
        \node[below] at (2,0) {$3$};
        \node[below] at (3,0) {$5$};
        &
        \filldraw (0,0) circle (0.05cm);
        \filldraw (1,0) circle (0.05cm);
        \filldraw (2,0) circle (0.05cm);
        \filldraw (3,0) circle (0.05cm);
        \filldraw (4,0) circle (0.05cm);
        \draw (0,0) -- (2,0);
        \draw (2,0) -- node[above] {$6$} (3,0);
        \draw (3,0) -- node[above] {$4$} (4,0);
        \node[below] at (3,0) {$5$};
        \node[below] at (4,0) {$3$};
        \\
        };
    \end{tikzpicture}
    \begin{tikzpicture}
        \matrix[matrix of nodes, column sep=30pt,row sep=6pt, nodes={anchor=center, minimum height=0.1cm, align=flush center}]{
        \begin{tikzpicture}
        \filldraw (0,0) circle (0.05cm);
        \filldraw (1,0) circle (0.05cm);
        \filldraw (2,0) circle (0.05cm);
        \draw (0,0) --node[above] {$5$} (1,0); 
        \draw (1,0) --node[above] {$5$} (2,0);
        
        \node[below] at (0,0) {$3$};
        \node[below] at (1,0) {$3$};
        \node[below] at (2,0) {$3$};
        \end{tikzpicture}
        &
        \begin{tikzpicture}
        \filldraw (0,0) circle (0.05cm);
        \filldraw (1,0) circle (0.05cm);
        \filldraw (0,1) circle (0.05cm);
        \filldraw (1,1) circle (0.05cm);
        \draw (0,0) --(1,0);
        \draw (0,1) --(1,1);
        \draw (0,0) --(0,1);
        \draw (1,0) --(1,1);
        \node[above left=0.001cm] at (0,1) {$3$};
        \node[below left=0.001cm] at (0,0) {$3$};
        \node[above right=0.001cm] at (1,1) {$3$};
        \node[below right=0.001cm] at (1,0) {$3$};
        
        \end{tikzpicture} 
        &
        \begin{tikzpicture}
        \filldraw (0,0) circle (0.05cm);
        \filldraw (1,0) circle (0.05cm);
        \filldraw (0,1) circle (0.05cm);
        \filldraw (1,1) circle (0.05cm);
        \draw (0,0) -- node[below] {$4$} (1,0);
        \draw (0,1) -- node[above] {$4$} (1,1);
        \draw (0,0) -- (0,1);
        \draw (1,0) -- (1,1);
        \node[above left=0.001cm] at (0,1) {$3$};
        \node[below left=0.001cm] at (0,0) {$3$};
        
        \end{tikzpicture} 
        &
        \begin{tikzpicture}
        \filldraw (0,0) circle (0.05cm);
        \filldraw (1,0) circle (0.05cm);
        \filldraw (-0.7,0.7) circle (0.05cm);
        \filldraw (-0.7,-0.7) circle (0.05cm);
        \draw (0,0) -- (-0.7,0.7);
        \draw (0,0) -- (-0.7,-0.7);
        \draw (0,0) --node[above] {$4$} (1,0);
        \node[below] at (0,0) {$3$};
        \node[below] at (-0.7,-0.7) {$3$};
        
        \node[above] at (-0.7,0.7) {$3$};
        
        \end{tikzpicture}
        \\
        };
    \end{tikzpicture}
    \caption{Examples of $\mathrm{CAT}(0)$ Shephard groups}
    \label{tab:excat0}
\end{table}

\begin{table}[p]
    \centering

    \vspace{3em}
    \begin{tikzpicture}
        \filldraw (0,0) circle (0.05cm);
        \filldraw (1,0) circle (0.05cm);
        \filldraw (2,0) circle (0.05cm);
        \filldraw (4,0) circle (0.05cm);
        \filldraw (5,0) circle (0.05cm);
        \filldraw (6,0) circle (0.05cm);
        \draw (0,0) -- node[above] {$\infty$} (1,0);
        \draw (1,0) -- (5,0);
        \draw (5,0) -- node[above] {$4$} (6,0);
        \node[fill=white,circle] (dots) at (3,0) {$\cdots$};
        \node[below=0.1cm] at (0,0) {$p$};
        \node[below=0.1cm] at (6,0) {$q$};
        \node at (7.5,0) {$p,q \in \mathbb{Z}_{\geq 2}$};
    \end{tikzpicture}

    \vspace{1em}

    \begin{tikzpicture}
        \filldraw (0,0) circle (0.05cm);
        \filldraw (1,0) circle (0.05cm);
        \filldraw (2,0) circle (0.05cm);
        \filldraw (4,0) circle (0.05cm);
        \filldraw (5,0) circle (0.05cm);
        \filldraw (6,0.5) circle (0.05cm);
        \filldraw (6,-0.5) circle (0.05cm);
        \draw (0,0) -- node[above] {$\infty$} (1,0);
        \draw (1,0) -- (5,0);
        \draw (5,0) -- (6,0.5);
        \draw (5,0) -- (6,-0.5);
        \node[fill=white,circle] (dots) at (3,0) {$\cdots$};
        \node[below=0.1cm] at (0,0) {$p$};
        \node at (7.5,0) {$p \in \mathbb{Z}_{\geq 2}$};
    \end{tikzpicture}

    \vspace{1em}

    \begin{tikzpicture}
        \filldraw (0,0) circle (0.05cm);
        \filldraw (1,0) circle (0.05cm);
        \filldraw (-1,0) circle (0.05cm);
        \draw (0,0) --node[above] {$4$} (1,0);
        \draw (0,0) --node[above] {$\infty$} (-1,0);
        \node[below] at (-1,0) {$p$};
        \node[below] at (0,0) {$q$};
        \filldraw ( 2,0) circle (0.05cm);
        \filldraw (-2,0) circle (0.05cm);
        \filldraw (-3,0) circle (0.05cm);
        \draw (2,0) -- (1,0);
        \draw (-2,0) --node[above] {$4$} (-1,0);
        \filldraw ( 3.5,0) circle (0.05cm);
        \filldraw (-4.5,0) circle (0.05cm);
        \draw (2,0) -- (3.5,0);
        \draw (-2,0) -- (-4.5,0);
        \node[fill=white,circle] (dots) at (2.75,0) {$\cdots$};
        \node[fill=white,circle] (dots) at (-3.75,0) {$\cdots$};
        \node at (4.8,0) {$p,q \in \mathbb{Z}_{\geq 2}$};
    \end{tikzpicture}

    \begin{tikzpicture}
        \matrix[column sep = 20pt]{
        \filldraw (0,0) circle (0.05cm);
        \filldraw (1,0) circle (0.05cm);
        \filldraw (0.5,0.9) circle (0.05cm);
        \draw (0,0) --node[below] {$\infty$} (1,0) -- (0.5,0.9) -- (0,0);
        \node[above] at (0.5,0.9) {$p$};
        \node[below left] at (0,0) {$p$};
        \node[below right] at (1,0) {$p$};
        \node at (2.2,0.5) {$p = 2,3,4,5$};
    &
        \filldraw (0,0) circle (0.05cm);
        \filldraw (1,0) circle (0.05cm);
        \filldraw (0,1) circle (0.05cm);
        \filldraw (1,1) circle (0.05cm);
        \draw (0,0) --(1,0);
        \draw (0,1) --node[above] {$\infty$} (1,1);
        \draw (0,0) --(0,1);
        \draw (1,0) -- (1,1);
        \node[above left=0.001cm] at (0,1) {$3$};
        \node[below left=0.001cm] at (0,0) {$3$};
        \node[above right=0.001cm] at (1,1) {$3$};
        \node[below right=0.001cm] at (1,0) {$3$};
    &
        \filldraw (0,0) circle (0.05cm);
        \filldraw (1,0) circle (0.05cm);
        \filldraw (0,1) circle (0.05cm);
        \filldraw (1,1) circle (0.05cm);
        \draw (0,0) -- (1,0);
        \draw (0,1) --node[above] {$\infty$} (1,1);
        \draw (0,0) --node[left] {$4$} (0,1);
        \draw (1,0) -- (1,1);
        
        \node[below left=0.001cm] at (0,0) {$3$};
        \node[above right=0.001cm] at (1,1) {$3$};
        \node[below right=0.001cm] at (1,0) {$3$};
    &
        \filldraw (0,0) circle (0.05cm);
        \filldraw (1,0) circle (0.05cm);
        \filldraw (0,1) circle (0.05cm);
        \filldraw (1,1) circle (0.05cm);
        \draw (0,0) -- (1,0);
        \draw (0,1) --node[above] {$\infty$} (1,1);
        \draw (0,0) --node[left] {$4$} (0,1);
        \draw (1,0) -- (1,1);
        \node[above left=0.001cm] at (0,1) {$p$};
        \node at (2,0.5) {$p \in \mathbb{Z}_{\geq 2}$};
        \\
        };
    \end{tikzpicture}
    \caption{Examples of cubulated Shephard groups}
    \label{tab:excub}
\end{table}

Recall that a Coxeter diagram is type (FC) if it satisfies the following condition:
\begin{enumerate}
    \item[(FC)] A subdiagram $\Gamma'$ of $\Gamma$ generates a finite Coxeter group if and only if $\Gamma'$ has no edge labeled $\infty$.
\end{enumerate}
Thus an equivalent hypothesis for Theorem \ref{thm:cube} is ``$\Gamma$ is an extended Coxeter diagram such that a subdiagram $\Gamma'$ of $\Gamma$ generates a finite Shephard group $G_{\Gamma'}$ if and only if $\Gamma'$ has no edge labeled $\infty$.''

We observe that Theorem \ref{thm:cube} includes the right-angled Shephard groups (with finite vertex labels) as a special case. More examples can be found in Table \ref{tab:excub}.

The inspiration for the definition of $\Theta_\Gamma$ comes from \cite{charney1995k}, and the overarching themes are quite similar: after detailing the construction of the complexes for a given Shephard group and the metrics thereupon, we verify that links in these complexes are $\mathrm{CAT}(1)$. The metrics are defined identically to the Coxeter and Artin group cases---however, our technique for showing that these metrics are $\mathrm{CAT}(0)$ diverge greatly from \cite{charney1995k}, heavily utilizing the connection between the finite Shephard groups and regular complex polytopes.

\begin{remark}
    We would like to briefly note that there is some similarity between the construction and results presented here with a recent paper of Soergel \cite{soergel2022generalization} which studies a subclass of Shephard groups called \emph{Dyer groups}. These ideas were developed independently, and there are many interesting Shephard groups covered by one result but not the other (although there is some overlap). For example, in \cite{soergel2022generalization}, some Shephard groups with vertex labels $\infty$ are shown to be $\mathrm{CAT}(0)$, whereas in this paper it is required that all vertices have finite labels. But in contrast, we can treat many Shephard groups which have finite-label edges between vertices whose labels are not $2$, which is not covered by the results of \cite{soergel2022generalization}. Moreover, the proof techniques remain largely independent.
\end{remark}

\subsection{Organization of paper}

In Section \ref{sec:finiteshephard}, we establish the classification of finite Shephard groups and recall how they relate to the regular complex polytopes. 
In Section \ref{sec:complexes}, we describe the complexes for Shephard groups which generalize the Coxeter and Deligne complexes of \cite{charney1995k}.
Section \ref{sec:cubical} is dedicated to proving Theorem \ref{thm:cube}, which only uses the geometry and combinatorics of the regular complex polytopes. 
Sections \ref{sec:cat1crit} and \ref{sec:moussong} are dedicated to proving Theorem \ref{thm:moussong}; Section \ref{sec:cat1crit} focuses solely on proving Theorem \ref{thm:firstCharney}, while Section \ref{sec:moussong} completes the proof of Theorem \ref{thm:moussong}. As an aside, Appendix \ref{sec:hessiancomb} compiles the various facts about the complex polytope for the ``$A_3(3)$'' Shephard group which are used in the proof of Theorem \ref{thm:moussong}.

\subsection{Acknowledgements}

The author would like to thank Mike Davis and Jingyin Huang for their many helpful discussions, and the anonymous referee for their comments. 

\section{Finite Shephard groups}\label{sec:finiteshephard}

In this section, we collect relevant background information regarding the finite Shephard groups, in particular, their classification and connection to regular complex polytopes.

\subsection{The classification of finite Shephard groups}

As discussed in the introduction, the study of Shephard groups has tended to focus on their use as finite complex reflection groups. Of course, it is well known (and one of their defining features) that the (non-Coxeter) Shephard groups arising as complex reflection groups all have a presentation as in (\ref{def:sheppres}) with diagram found in Table \ref{tab:finiteshephards}. However,
if one considers an abstract group with presentation (\ref{def:sheppres}) and arbitrary diagram $\Gamma$, it is not clear if the converse holds; that is, if $G_\Gamma$ is finite, then $\Gamma$ appears in Tables \ref{tab:finitediagrams} or \ref{tab:finiteshephards}. It turns out that this is in fact true, although not entirely trivial to show.
To the author's knowledge, there is no published full proof of this fact (as previous work seems to only rely on the fact that the diagrams in Table \ref{tab:finiteshephards} form a subset of the finite Shephard groups), so, we include a brief section on the proof that the diagrams in Tables \ref{tab:finitediagrams} and \ref{tab:finiteshephards} form the entire class of finite Shephard groups.

The following definition is adapted from \cite{mostow1980remarkable}, who adapted it from the Hermitian forms typically used in the study of classical Shephard groups.

\begin{defn}
    Let $\Gamma$ be a Shephard group on vertex set $I$ with all $p_i$ finite. For an edge spanned by $i,j \in I$, $i \neq j$, define
    \begin{align*}
        \alpha_{ij} = - \left(  \frac{ \cos(\pi/p_i - \pi/p_j) + \cos(2\pi/m_{ij}) }{ 2 \sin(\pi/p_i)\sin(\pi/p_j)}\right)^{1/2}.
    \end{align*}
    We let $\alpha_{ii} = 1$. 
    If $i$ and $j$ are not joined by an edge, then $\alpha_{ij} = 0$. 

    We define a Hermitian form $H = H_\Gamma = \langle \cdot, \cdot \rangle$ on $\mathbb{C}^I$ with basis $\{e_i:i \in I\,\}$ by
    \begin{align*}
        \langle e_i,e_j \rangle = \alpha_{ij}.
    \end{align*}  
    The form $H$ viewed as a matrix (which we still call $H$ by abuse of notation) is analogous to the cosine matrix for a usual Coxeter group, and reduces to the cosine matrix if all $p_i = 2$. 
    We can then define a representation $\rho = \rho_\Gamma$ of $G_\Gamma$ by sending the generator $s_i$ to the complex reflection
    \begin{align*}
        R_i(z) = z + (\zeta_{p_i} - 1) \langle z, e_i \rangle e_i,
    \end{align*}
    where $\zeta_{p_i} = \exp(2\pi \sqrt{-1}/p_i)$ is the usual primitive $p_i$-th root of unity. 
    (This reduces to the complexification of the usual ``canonical representation'' of a Coxeter group if $\Gamma$ is Coxeter.)
\end{defn}

Our main tool is the following observation: 

\begin{prop} \label{prop:finimpliespos}
    If the Shephard group $G_\Gamma$ is finite, then the Hermitian form $H_\Gamma$ is positive definite.
\end{prop}

The proof is identical to the corresponding statement for Coxeter groups. We provide the proof here for completeness. First, we need a lemma.

\begin{lemma} \label{lem:kernelofH} \emph{(cf. \cite[Ch.~V, \S 4.6 Thm.~7]{bourbaki2008lie})}
    Suppose $\Gamma$ is connected. Let $H^0 = H^0_\Gamma$ denote the kernel of $H = H_\Gamma$, i.e.,
    \[
        H^0 = \{\,x \in \mathbb{C}^I:H(x,y) = 0 \text{ \emph{for each} } y \in \mathbb{C}^I\,\}.
    \]
    Then $G_\Gamma$ acts trivially on $H^0$ and every proper $G_\Gamma$-invariant subspace is contained in $H^0$.
\end{lemma}

\begin{proof}
    By the definition of the maps $R_i$, if $z \in H^0$ then $R_i(z) = z$ for all $i$, and hence $H^0$ is fixed by $G_\Gamma$.

    Let $E$ be a proper subspace of $\mathbb{C}^I$ fixed by $G_\Gamma$. Suppose some $e_i$ is contained in $E$. Let $j$ be a vertex of $\Gamma$ joined to $i$ by an edge. By direct computation we see that the coefficient of $e_j$ in $R_j(e_i)$ is nonzero, and hence $e_j$ is also contained in $E$. Since $\Gamma$ is connected, it follows that every basis vector is contained in $E$, contradicting our assumption that $E$ is proper. Thus $E$ contains no basis vector.

    Consider the action of some $R_i$ on $\mathbb{C}^I$. By basic linear algebra, $\mathbb{C}^I$ decomposes as the $+1$ and $\zeta_{p_i}$ eigenspaces of $R_i$, which are $e_i^\perp$ (the subspace orthogonal to $e_i$ under $H$) and $\mathbb{C} e_i$, respectively. Since $e_i \not\in E$ by assumption, we must have $E \subseteq e_i^\perp$. Thus $E \subseteq \bigcap e_i^\perp = H^0$.
\end{proof}

We can now complete the proof of Proposition \ref{prop:finimpliespos}.

\begin{proof}[Proof (of Proposition \ref{prop:finimpliespos})]
    Without loss of generality, we may assume $\Gamma$ is connected (otherwise $H_\Gamma$ is positive definite if and only if the Hermitian form for each component is positive definite). 

    First we note that by classical representation theory\footnote{Finite group representations $\tau : H \to \mathbb{C}$ are semisimple by Mashke's theorem \cite{serre1977linear}, and fix the inner product defined by $(x,y) = \frac{1}{|H|} \sum_{h \in H} \langle \tau(h)x, \tau(h)y \rangle$, where $\langle \cdot,\cdot \rangle$ is the standard inner product on $\mathbb{C}^n$.}, $\rho$ is semisimple (i.e., if a subspace of $\mathbb{C}^I$ is $G$-invariant under $\rho$ then it's a direct summand of $\mathbb{C}^I$) and fixes some inner product (i.e., a positive definite Hermitian form). 
    Since this representation is semisimple, Lemma \ref{lem:kernelofH} implies $H$ is non-degenerate; 
    otherwise, we would have that $H^0 \neq 0$ is a proper non-trivial subspace fixed by $G$, but by the Lemma, there can be no complementary $G$-invariant subspace of $H^0$, contradicting the fact that $\rho$ is semisimple. Thus $H$ is nondegenerate, i.e., has $H^0 = 0$, and thus by Lemma \ref{lem:kernelofH}, we know that the representation $\rho$ is irreducible.
    Since $\rho$ is irreducible, it follows from Schur's lemma that $H$ is a scalar multiple of the aforementioned invariant inner product. Since we've required $H(e_i,e_i) = 1$ for all $i$, it follows that $H$ is positive definite.
\end{proof}

Our main use for this proposition is via its contrapositive, namely that if $H_\Gamma$ is not positive definite, then $G_\Gamma$ is infinite. The following is the other key technical fact.

\begin{lemma} \label{lem:2gencomp}
    Let $p,q \geq 2$ and $m \geq 3$ be integers with $p = q$ whenever $m$ is odd. Define
    \begin{align*}
        c &= -\cos(\pi/m), \\
        \alpha &= - \left(  \frac{ \cos(\pi/p - \pi/q) + \cos(2\pi/m) }{ 2 \sin(\pi/p)\sin(\pi/q)}\right)^{1/2}.
    \end{align*}
    Then $\alpha \leq c$.
\end{lemma}
\begin{proof}

    We'll consider two cases. First, suppose $\alpha \leq -1$. Since cosine is bounded above by 1, this means 
    \begin{align*}
        \alpha \leq -1 \leq -\cos(\pi/m) = c.
    \end{align*}

    Now suppose $\alpha > -1$. Since $\alpha \leq 0$, this implies $\alpha^2 < 1$. By a straightforward trigonometric exercise, this happens if and only if $ 1/p + 1/q  +  2/m > 1$.
    There are only finitely many such triples $(p,m,q)$ when
    $m \neq 4$
    (since $p = q$ when $m = 3$). These cases are easily checked by direct computation.
    When $m = 4$, this relation holds only when either $p=q=3$, or $q = 2$ and $p \geq 2$ (or vice versa). If $p=q=3$ then the computation is immediate. If $q = 2$ and $p \geq 2$, we see that 
    \begin{align*}
        \alpha 
        &= -\left( \frac{\cos(\pi/p - \pi/2) + \cos(2\pi/4)}{2\sin(\pi/p)\sin(\pi/2)}\right)^{1/2} \\
        &= -\left( \frac{\sin(\pi/p)}{2\sin(\pi/p)}\right)^{1/2} \\
        &= -2^{-1/2} \\
        &= -\cos(\pi/4) = c. \qedhere
    \end{align*}
\end{proof}

\begin{prop} \label{lem:extendfiniteunderlying}
    Let $\Gamma$ be an extended Coxeter diagram. If the Coxeter group $W_\Gamma$ is not finite, then neither is $G_\Gamma$. (In other words, $\mathcal S^{fs} \subseteq \mathcal S^f$.)
\end{prop}

\begin{proof}
    If $W_\Gamma$ is not finite, then the cosine matrix $C = C_\Gamma = (c_{ij})$, where $c_{ij} = - \cos(\pi/m_{ij})$, $c_{ii} = 1$ 
    is not positive definite \cite[\S 4.8 Thm.~2]{bourbaki2008lie}.
    This means there exists a non-zero vector $x \in \mathbb{C}^n$ such that 
    \begin{align*}
        x^* C x \leq 0.
    \end{align*}
    Expanding the product gives us
    \begin{align*}
        \sum c_{ij} \overline{x_i} x_j \leq 0.
    \end{align*}
    By Lemma \ref{lem:2gencomp}, in combination with the requirements that $\alpha_{ij} = 0$ when $i$ and $j$ don't span an edge of $\Gamma$ and $\alpha_{ii} = 1$, we know that $\alpha_{ij} \leq c_{ij}$ for all $i,j$. Therefore,
    \begin{align*}
        x^* H x = \sum \alpha_{ij} \overline{x_i}x_j \leq \sum c_{ij} \overline{x_i} x_j \leq 0.
    \end{align*}
    Hence $H = H_\Gamma$ is not positive definite, so $G_{\Gamma}$ is not finite.
\end{proof}

This narrows down our search to adding vertex labels to the diagrams in Table \ref{tab:finitediagrams}. 
The following Lemma rules out the branched Coxeter diagrams.

\begin{lemma} \label{lem:dpnonfinite}
    Let $D_4(p)$ denote the following extended Coxeter diagram:
    
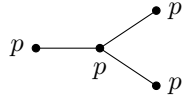
\begin{center}
    \begin{tikzpicture}
        \filldraw (0.15,0) circle (0.05cm);
        \filldraw (1.75,0.5) circle (0.05cm);
        \filldraw (1.75,-0.5) circle (0.05cm);
        \draw (0.15,0) -- (1,0);
        \filldraw (1,0) circle (0.05cm);
        \draw (1.75,0.5) -- (1,0) -- (1.75,-0.5); 
        \node at (-0.1,0) {$p$};
        \node at (1,-0.3) {$p$};
        \node at (2,0.5) {$p$};
        \node at (2,-0.5) {$p$};
    \end{tikzpicture}
    \end{center}
    If $p > 2$, then the corresponding Hermitian form $H_{D_4(p)}$ is not positive definite. In particular, any Shephard group with a subdiagram of this form is not finite. 
\end{lemma}
\begin{proof}
    Let $H = H_{D_4(p)}$. Suppose the vertex set of $D_4(p)$ is $\{1,2,3,4\}$, with $1$ being the ``central'' (valence 3) vertex.
    We compute for $j \neq 1$ that
    \begin{align*}
        \alpha_{1j} = \alpha_{j1} &= - \left(  \frac{ \cos(\pi/p - \pi/p) + \cos(2\pi/3) }{ 2 \sin(\pi/p)\sin(\pi/p)}\right)^{1/2}\\
        &= - \left(\frac{1 + \cos(2\pi/3)}{2 \sin^2(\pi/p)}\right)^{1/2} \\
        &= - \left(\frac{1 - \frac{1}{2}}{2 \sin^2(\pi/p)}\right)^{1/2} \\
        &= - \left(\frac{1 }{4 \sin^2(\pi/p)}\right)^{1/2} \\
        &= - \frac{1}{2\sin(\pi/p)}.
    \end{align*}
    By the definition of $H$, for all $i$, we have $\alpha_{ii} = 1$, and if $i \neq j$ and neither $i$ nor $j$ are $1$, then $\alpha_{ij} = 0$. Thus the matrix representing $H$ is 
    \begin{align*}
        \left( \begin{matrix}
            1 & -(2\sin(\pi/p))^{-1}& -(2\sin(\pi/p))^{-1} & -(2\sin(\pi/p))^{-1} \\
            -(2\sin(\pi/p))^{-1} & 1 & 0 & 0 \\
            -(2\sin(\pi/p))^{-1} & 0 & 1 & 0 \\
            -(2\sin(\pi/p))^{-1} & 0 & 0 & 1 \\
        \end{matrix} \right).
    \end{align*}   
    The determinant of this matrix is 
    \[
        1 - \frac{3}{4\sin^2(\pi/p)}
    \]
    Notice that this function is decreasing in $p$. When $p = 3$, this determinant is $0$. Hence it is non-positive for all $p \geq 3$, and thus $H$ is not positive definite when $p \geq 3$. 
    It follows that any Shephard group with a $D_4(p)$ subdiagram has Hermitian form which is not positive definite, so the group is not finite.
\end{proof}

\begin{thm} \label{thm:classifyfinite}
    A connected extended Coxeter diagram $\Gamma$ produces a finite Shephard group $G_\Gamma$ if and only if it appears in Tables \ref{tab:finitediagrams} or \ref{tab:finiteshephards}. 
\end{thm}
\begin{proof}
    The result for Coxeter diagrams is well known \cite{bourbaki2008lie}. So, suppose $\Gamma$ is non-Coxeter (i.e., has at least one vertex with label $> 2$). If $\Gamma$ is a straight line, then by \cite[\S 13]{1975regular}, $G_\Gamma$ is finite if and only if it appears in Table \ref{tab:finiteshephards}. If $\Gamma$ is not a straight line, then, by Proposition \ref{lem:extendfiniteunderlying}, it must have underlying Coxeter diagram of type $D_n$, $n \geq 4$, or $E_m$, $m = 6,7,8$. Since the label of vertices connected by an odd-labeled edge must agree, we have that every vertex label of $\Gamma$ is $p$ for some $p \geq 2$. Since we've assumed $\Gamma$ is non-Coxeter, we know $p \geq 3$. But then $\Gamma$ must contain a subdiagram of the form $D_4(p)$ with $p \geq 3$; by Lemma \ref{lem:dpnonfinite}, this forces the group to be infinite.
\end{proof}

\begin{table}[ht]
    \centering
    \begin{minipage}{\linewidth}
    \setlength{\columnsep}{-2cm}
        \begin{multicols}{2}
            \begin{tikzpicture}
                \node at (-0.75,0) {$A_n$};  
                \filldraw (0,0) circle (0.05cm);
                \draw (0,0) -- (2.3,0);
                \filldraw (1,0) circle (0.05cm);
                \filldraw (2,0) circle (0.05cm);
                \node at (2.75,0) {. . .};
                \filldraw(3.5,0) circle (0.05cm);
                \draw (3.5,0) -- (3.25,0);
            \end{tikzpicture}

            \begin{tikzpicture}
                \node at (-0.75,0) {$B_n$}; 
        \filldraw (0,0) circle (0.05cm);
        \draw (0,0) -- node[above] {$4$} (1,0);
        \draw (1,0) -- (2.3,0);
        \filldraw (1,0) circle (0.05cm);
        \filldraw (2,0) circle (0.05cm);
        \node at (2.75,0) {. . .};
        \filldraw(3.5,0) circle (0.05cm);
        \draw (3.5,0) -- (3.25,0);
        \end{tikzpicture}
        \vspace{0.5em}

        \begin{tikzpicture}
             \node at (-0.75,0) {$D_n$}; 
        \filldraw (0,0) circle (0.05cm);
        \draw (0,0) -- (1.3,0);
        \filldraw (1,0) circle (0.05cm);
        \node at (1.75,0) {. . .};
        \filldraw(2.5,0) circle (0.05cm);
        \draw (2.5,0) -- (2.25,0);
        \filldraw (3.45,0.5) circle (0.05cm);
        \filldraw (3.45,-0.5) circle (0.05cm);
        \draw (3.45,0.5) -- (2.5,0) -- (3.45,-0.5); 
        \end{tikzpicture}
        
        \begin{tikzpicture}
        \node at (-0.75,0) {$I_2(m)$}; 
        \filldraw (0,0) circle (0.05cm);
        \draw (0,0) -- node[above] {$m$} (1,0);
        \filldraw (1,0) circle (0.05cm);
        \end{tikzpicture}
        
        \begin{tikzpicture}
        \node at (-0.75,0) {$H_3$}; 
        \filldraw (0,0) circle (0.05cm);
        \draw (0,0) -- node[above] {$5$} (1,0);
        \filldraw (1,0) circle (0.05cm);
        \draw (1,0) -- (2,0);
        \filldraw (2,0) circle (0.05cm);
        \end{tikzpicture}
        
        \begin{tikzpicture}
        \node at (-0.75,0) {$H_4$};
        \filldraw (0,0) circle (0.05cm);
        \draw (0,0) -- node[above] {$5$} (1,0);
        \filldraw (1,0) circle (0.05cm);
        \draw (1,0) -- (2,0);
        \filldraw (2,0) circle (0.05cm);
        \draw (3,0) -- (2,0);
        \filldraw (3,0) circle (0.05cm);
        \end{tikzpicture}
        \vspace{1em}
        
        \begin{tikzpicture}
        \node at (-0.75,0) {$F_4$};
        \filldraw (0,0) circle (0.05cm);
        \draw (0,0) -- (1,0);
        \filldraw (1,0) circle (0.05cm);
        \draw (1,0) -- node[above] {$4$} (2,0);
        \filldraw (2,0) circle (0.05cm);
        \draw (3,0) -- (2,0);
        \filldraw (3,0) circle (0.05cm);
        \end{tikzpicture}
        \vspace{1em}
        
        \begin{tikzpicture}
        \node at (-0.75,-0.5) {$E_6$};
        \filldraw (0,0) circle (0.05cm);
        \filldraw (1,0) circle (0.05cm);
        \filldraw (2,0) circle (0.05cm);
        \filldraw (3,0) circle (0.05cm);
        \filldraw (4,0) circle (0.05cm);
        \filldraw (2,-1) circle (0.05cm);
        \draw (0,0) -- (4,0);
        \draw (2,0) -- (2,-1);
        \end{tikzpicture}
        \vspace{1em}
        
        \begin{tikzpicture}
        \node at (-0.75,-0.5) {$E_7$};
        \filldraw (0,0) circle (0.05cm);
        \filldraw (1,0) circle (0.05cm);
        \filldraw (2,0) circle (0.05cm);
        \filldraw (3,0) circle (0.05cm);
        \filldraw (4,0) circle (0.05cm);
        \filldraw (5,0) circle (0.05cm);
        \filldraw (2,-1) circle (0.05cm);
        \draw (0,0) -- (5,0);
        \draw (2,0) -- (2,-1);
        \end{tikzpicture}
        \vspace{1em}
        
        \begin{tikzpicture}
    \node at (-0.75,-0.5) {$E_8$};
        \filldraw (0,0) circle (0.05cm);
        \filldraw (1,0) circle (0.05cm);
        \filldraw (2,0) circle (0.05cm);
        \filldraw (3,0) circle (0.05cm);
        \filldraw (4,0) circle (0.05cm);
        \filldraw (5,0) circle (0.05cm);
        \filldraw (6,0) circle (0.05cm);
        \filldraw (2,-1) circle (0.05cm);
        \draw (0,0) -- (6,0);
        \draw (2,0) -- (2,-1);
        \end{tikzpicture}
        \vspace{1em}

        \end{multicols}
    \end{minipage}
    \vspace{1em}

    \caption{The finite Coxeter diagrams}
    \label{tab:finitediagrams}

    \vspace{-1em}
\end{table}

\begin{table}[ht]
    \centering
    \begin{minipage}{\linewidth}
    \setlength{\columnsep}{-0.5cm}
        \begin{multicols}{2}
        \begin{tikzpicture}
            \node at (-1.35,0) {$B_n(p,2)$}; 
            \filldraw (0,0) circle (0.05cm);
            \draw (0,0) -- node[above] {$4$} (1,0);
            \draw (1,0) -- (2.3,0);
            \filldraw (1,0) circle (0.05cm);
            \filldraw (2,0) circle (0.05cm);
            \node at (2.75,0) {. . .};
            \filldraw(3.5,0) circle (0.05cm);
            \draw (3.5,0) -- (3.25,0);
            \node at (0,-0.25) { $p$ };
        \end{tikzpicture}

        \begin{tikzpicture}
        \node at (-1.36,0) {$B_3(2,3)$}; 
        \filldraw (0,0) circle (0.05cm);
        \draw (0,0) -- node[above] {$4$} (1,0);
        \filldraw (1,0) circle (0.05cm);
        \draw (1,0) -- (2,0);
        \filldraw (2,0) circle (0.05cm);
        \node at (1,-0.3) {$3$};
        \node at (2,-0.3) {$3$};
        \end{tikzpicture}

        \begin{tikzpicture}
        \node at (-1.2,0) {$I_2(p,m,q)$}; 
        \filldraw (0,0) circle (0.05cm);
        \draw (0,0) -- node[above] {$m$} (1,0);
        \filldraw (1,0) circle (0.05cm);
        \node at (0,-0.25) {$p$};
        \node at (1,-0.25) {$q$};
        \node at (2.8,0) { $ \left( \frac{1}{p} + \frac{1}{q} + \frac{2}{m} > 1 \right)$ };
        \end{tikzpicture}

        \vspace{6em}

        \begin{tikzpicture}
        \node at (-1,0) {$A_3(3)$}; 
        \filldraw (0,0) circle (0.05cm);
        \draw (0,0) -- (1,0);
        \filldraw (1,0) circle (0.05cm);
        \draw (1,0) -- (2,0);
        \filldraw (2,0) circle (0.05cm);
        \node at (0,-0.3) {$3$};
        \node at (1,-0.3) {$3$};
        \node at (2,-0.3) {$3$};
        \end{tikzpicture}
        
        \begin{tikzpicture}
        \node at (-1,0) {$A_4(3)$};
        \filldraw (0,0) circle (0.05cm);
        \draw (0,0) -- (1,0);
        \filldraw (1,0) circle (0.05cm);
        \draw (1,0) -- (2,0);
        \filldraw (2,0) circle (0.05cm);
        \draw (3,0) -- (2,0);
        \filldraw (3,0) circle (0.05cm);
        \node at (0,-0.3) {$3$};
        \node at (1,-0.3) {$3$};
        \node at (2,-0.3) {$3$};
        \node at (3,-0.3) {$3$};
        \end{tikzpicture}

        \end{multicols}
        \end{minipage}
        
    \vspace{1em}

    \caption{The non-Coxeter finite extended Coxeter diagrams}
    \label{tab:finiteshephards}

    \vspace{-1em}
\end{table}

\begin{table}[ht]
    \begin{tabular}{l|cl}
        & Shephard-Todd & \quad Coxeter \\ \hline
        $B_n(p,2)$ & $G(p,1,n)$ & $p[4]2[3]...[3]2$ \\
        $B_3(2,3)$ & $G_{26}$ &  $2[4]3[3]3$ \\
        $A_3(3)$ & $G_{25}$ & $3[3]3[3]3$ \\
        $A_4(3)$ & $G_{32}$ & $3[3]3[3]3[3]3$ \\
        $I_2(p,m,q)$ & various & $p[m]q$
    \end{tabular}
    \vspace{1em}

    \caption{Comparison of notation}
    \label{tab:shephardNotation}
\end{table}

\subsection{Connection to complex polytopes}
\label{sec:connectpoly}

We end this section by noting the deep relationship between the so-called regular complex polytopes and the (finite) Shephard groups, as this forms the basis for most of our arguments. First, we recall the definitions. In many ways they are identical to (reasonable) definitions of regular polytopes in $\mathbb{R}^d$, with the only substantial and necessary differences made to accomodate the fact that complex ``lines'' have real dimension 2. A standard reference for this material is \cite{1975regular}.

We introduce common terminology. Let $\mathcal P$ be any collection of affine subspaces of $\mathbb{C}^d$ (which may include the ``empty subspace'' $\varnothing$ and the whole space).  
An \emph{$n$-face} of $\mathcal P$ is an $n$-dimensional element of $\mathcal P$ (by convention, $\mathrm{dim}(\varnothing) = -1$). 
Sometimes we call $0$-faces the vertices of $\mathcal P$ and $1$-faces the edges of $\mathcal P$.
We say two faces are incident if they are distinct and one contains the other as a subspace.
If $F_1 \subseteq F_2$ are incident faces and $\dim(F_1) \leq \dim(F_2) - 2$, then the collection of all $F'$ with $F_1 \subseteq F' \subseteq F_2$ is called the \emph{medial figure} of $F_1$ and $F_2$. 
A symmetry of $\mathcal P$ is a unitary map of $\mathbb{C}^d$ which fixes the union of the elements of $\mathcal P$ as a set. The collection of all such symmetries is the \emph{symmetry group} of $\mathcal P$. Note that a symmetry necessarily preserves incidence.

\begin{defn}
    We call a collection $\mathcal P$ of affine subspaces of $\mathbb{C}^d$ a \emph{regular complex ($n$-)polytope} if it satisfies the following axioms.
    \begin{enumerate}
        \item $\varnothing \in \mathcal P$.
        \item There is unique $n$-face of $\mathcal P$ and all elements of $\mathcal P$ are contained in this $n$-face. 
        \item Every medial figure contains at least two faces of each (appropriate) dimension.
        \item Every medial figure of an $i$-face with an $j$-face with $i \leq j-3$ is connected.
        \item Its symmetry group acts simply transitively on the set of maximal chains (linearly ordered sets) of faces.
    \end{enumerate}
\end{defn}

A standard result is the following.

\begin{thm}
    If $\mathcal P$ is a regular complex polytope, then its symmetry group is a finite Shephard group with diagram a straight line (meaning the valence of exactly two of the vertices is $1$ and all other vertices have valence $2$). Conversely, if $\Gamma$ is a straight-line extended Coxeter diagram then $G_\Gamma$ is the symmetry group of a regular complex polytope. 
\end{thm}

This is detailed in \cite[\S 12.1]{1975regular}. This is not a 1-1 correspondence, as there are ``starry'' and ``non-starry'' polytopes that can be constructed for a given Shephard group. From here on, when we say ``regular complex polytope'' we always mean the non-starry polytope. With this restriction, the above does give a 1-1 correspondence (modulo taking dual polytopes).

We highlight specifically how a Shephard group acts on its corresponding polytope. 
Let $\Gamma$ be a connected extended Coxeter diagram which is linear and has vertices named $1,\dots,n$ from left to right (or right to left) with $\{s_i:i=1,\dots,n\,\}$ the generators of $G_\Gamma$.
We fix an arbitrary maximal chain $F_{-1} < F_0 < F_1 < \dots < F_n < F_{n+1}$ (where $F_{-1} = \varnothing$ and $F_{n+1}$ is the top-dimensional face of $\mathcal P$) sometimes denoted by $\mathcal F_0$ and called the \emph{base chain}\footnote{Some authors call $\mathcal F_0$ the ``base flag''---we use ``chain'' rather than ``flag'' as to not confuse the complex of chains with the notion of ``flag complex'' for an arbitrary simplicial complex.}.
A generator $s_i$ acts by fixing the elements $F_j$ for $j \neq i$, and cyclically permuting the $i$-faces which are incident to each $F_j$, $j \neq i$. Moreover,
the stabilizer of the face $F_i$ is the subgroup of $G$ generated by $\widehat{s_i} \coloneqq S \setminus \{s_i\}$.
The following clarifies the properties of this action further.

\begin{lemma} \label{lem:shepacting}
\ \emph{\cite[Lemma 12]{ORS2000TheSR}}
    Let $F_{i_1} \subseteq F_{i_2} \subseteq \dots \subseteq F_{i_k}$ be a subchain of $\mathcal F_0$, with $T = \{s_{i_1}, \dots, s_{i_k}\}$.
    The stabilizer of this chain is $G_{\widehat T}$, where $\widehat T = S \setminus T$.
\end{lemma}

We will call a $G$-translate of such a subchain a ``chain of type $\widehat T$'', since the stabilizer is a conguate of the subgroup of $G$ generated by $\widehat T$.

The following clarifies that these stabilizers are themselves Shephard groups generated by subdiagrams of $\Gamma$.

\begin{prop} \label{prop:shepdevel}
    Let $\Gamma$ be a finite extended Coxeter diagram on vertex set $I$ with $S = \{\, s_i : i \in I\,\}$. Let $T \subseteq S$. Then $G_T \cong G_{\Gamma(T)}$ via the natural map discussed in the introduction.
\end{prop}
\begin{proof}
    We give a rough outline of the argument which is implicit throughout \cite{1975regular}. 

    If $\Gamma$ is Coxeter the result is well known, so suppose it is one of the diagrams in Table \ref{tab:finiteshephards}. 
    Let $\Gamma'$ be a connected subdiagram. 
    Name the vertices of $\Gamma$ left to right as $1,\dots,n$, and let $T = \{s_i, s_{i+1},\dots,s_j\}$ for some $i \leq j$.
    Let 
    $\mathcal P_{i,j}$ denote the medial figure of $F_{i-1}$ and $F_{j+1}$. 
    A medial figure of a regular complex polytope from an $(i-1)$-face to a $(j+1)$-face is itself a regular complex polytope. Moreover, the type\footnote{We've glossed over this, but it is possible to combinatorially describe a polytope via an extended Coxeter diagram without reference to symmetry group, and this is what we mean by ``type''.} of this medial figure is the subdiagram of $\Gamma$ from $i$ to $j$, which in our case is $\Gamma(T)$. This implies that the symmetry group of $\mathcal P_{i,j}$ is $G_{\Gamma(T)}$. 
    But by the description of the action of $G_\Gamma$ given above, we see that $\mathcal P_{i,j}$ is the orbit of the base chain $\mathcal F_0$ under $G_T$ and, moreover, that this action is simply transitive. Thus $G_T$ is also the symmetry group of $\mathcal P_{i,j}$, and since the action of the generators of $G_T$ and $G_{\Gamma(T)}$ is identical, it can then be easily seen that $G_T \cong G_{\Gamma(T)}$.

    The case where $T$ is an arbitrary subset of $S$ follows immediately from seeing that $G_{\Gamma(T)}$ is a direct product of groups with connected diagram, and $G_T$ is an (internal) direct product of subsets of $S$ of the form considered above.
\end{proof}

Thus in the case when $G_\Gamma$ is finite, we freely identify $G_T$ and $G_{\Gamma(T)}$ and will refer to either as a (standard) parabolic subgroup. 

The following standard construction was used implicitly in the previous proof, and it will be useful to have the precise statement later. 

\begin{defn} \label{def:prodpoly}
    Let $\mathcal P_1, \mathcal P_2$ be regular complex polytopes in $\mathbb{C}^{d_1}$ and $\mathbb{C}^{d_2}$, respectively. Then the product polytope $\mathcal P = \mathcal P_1 \times \mathcal P_2$ is the collection of all subspaces $f_1 \times f_2 \subseteq \mathbb{C}^{d_1} \times \mathbb{C}^{d_2} \cong \mathbb{C}^{d_1d_2}$ with $f_i \in \mathcal P_i$, ordered by inclusion. Note that $e_1 \times e_2 \subseteq f_1 \times f_2$ if and only if both $e_i \subseteq f_i$.
\end{defn}

Repeated products are defined in the obvious way. We emphasize that for any subspace $V \subseteq \mathbb{C}^n$, we always have $\varnothing \times V = V \times \varnothing = \varnothing$. In particular, if $\mathcal P_{> \varnothing}$ denotes the non-empty faces of $\mathcal P$, then $\mathcal P_{> \varnothing} \times \mathcal Q_{> \varnothing} = (\mathcal P \times \mathcal Q)_{> \varnothing}$.

\section{Complexes for Shephard groups}
\label{sec:complexes}

Now we describe two complexes for a given Shephard group based on familiar definitions for Coxeter groups and Artin groups.

First, we introduce common notation: if $\Psi$ is a cell complex, we denote the closed (resp. open) star of a vertex $v$ in $\Psi$ by $\mathrm{St}(v) = \mathrm{St}(v,\Psi)$ (resp. $\mathrm{st}(v) = \mathrm{st}(v,\Psi)$), and the boundary of a star is $\partial \mathrm{St}(v) = \partial \mathrm{St}(v,\Psi) = \mathrm{St}(v) \setminus \mathrm{st}(v)$. The link of $v$ in $\Psi$ is denoted $lk(v) = lk(v,\Psi)$.

\subsection{A complex for finite Shephard groups}
In this section, let $\Gamma$ be an extended Coxeter diagram such that $G = G_\Gamma$ is finite. 
In \cite{orlik1990milnor}, an analogue of the Coxeter complex is defined for Shephard groups, there called the \emph{Milnor fiber complex}.
This complex is one of the main objects of study of this paper.
For non-Coxeter finite Shephard groups, it acts as a generalization of the Coxeter complex for reasons we explain shortly. 
We briefly recall the definition before restating it in terms of complexes of groups. (The original definition is given in terms of invariant polynomials, but is equivalent to the following definition in terms of polytopes by \cite[Thm.~5.1]{orlik1990milnor}.)

Suppose $\Gamma$ is a straight line diagram (recall: $\Gamma$ is connected and all vertices are valence 2 except for two unique vertices of valence 1). Let $\mathcal P$ be the (unique non-starry) regular complex $n$-polytope with $G_\Gamma$ as its symmetry group.
Denote by $\mathcal P'_p$ the derived complex of the subposet $\mathcal P_p$ of $\mathcal P$ consisting of the \emph{proper} faces of $\mathcal P$; that is, the set of chains (linearly ordered subsets) of $\mathcal P$ excluding $\varnothing$ and $\mathbb{C}^n$ (the unique $n$-dimensional face of $\mathcal P$). This is an abstract simplicial complex, and thus has a geometric realization, which we will call the \emph{Milnor fiber complex} or \emph{extended Coxeter complex} $\widehat \Theta = \widehat \Theta_\Gamma = \widehat \Theta(\Gamma)$. Note that if $\Gamma$ is a (straight line) Coxeter diagram then this is isomorphic to the usual Coxeter complex. 
The $n$-simplices of $\widehat \Theta$ will be denoted
\[
    [ f_0 <  f_1 < \dots < f_n]
\]
where $f_0 <  f_1 < \dots < f_n$ is a chain of faces of $\mathcal P$. To rephrase the definition of the derived complex, the vertices of $\widehat \Theta$ correspond to (proper) faces of $\mathcal P$, and a collection of vertices span a simplex if and only if they are linearly nested (which in turn happens if and only if they are pairwise nested, implying $\widehat \Theta$ is a flag complex). Thus $\widehat \Theta$ inherits an action of $G$ from its action on $\mathcal P$. There is a ``base chamber'' $\mathcal C_0$ of $\widehat \Theta$ corresponding to the base chain $\mathcal F_0$. If $F_{i_1} < F_{i_2} < \dots < F_{i_k}$ is a subchain of $\mathcal F_0$, with $T = \{s_{i_1}, \dots, s_{i_k}\}$, then we call a $G$-translate of 
\[
    [F_{i_1} < F_{i_2} < \dots < F_{i_k}]
\]
a \emph{simplex of type $\widehat T$}.
Since $G$ acts simply transitively on the set of chains of $\mathcal P$ by assumption, $\mathcal C_0$ is a strict fundamental domain for the action of $G$ on $\widehat \Theta$. 
By Lemma \ref{lem:shepacting}, the stabilizer of a simplex of type $\widehat T$ contained in the base chamber is $G_{\widehat T}$. This implies that the stabilizer of a simplex of type $\widehat T$ is a conjugate of $G_{\widehat T}$.

We can rephrase the above in the language of complexes of groups as follows. 
\begin{defn}
Let $\Gamma$ be any finite extended Coxeter diagram. Define $\Delta = \Delta_\Gamma = \Delta_S$ to be a simplex whose vertices are labeled by the generators $S$ of $G_\Gamma$ (equiv., the vertices of $\Gamma$). 
For $T \subseteq S$, let $\sigma_{T}$ denote the face of $\Delta_S$ spanned by the elements of $T$.
We define a complex of groups $\widehat{\mathcal G} = \widehat{\mathcal G}(G_\Gamma, \Delta_\Gamma)$ by declaring the local group at the face $\sigma_{T}$ to be the group $G_{\widehat T}$. (Recall $\widehat T \coloneqq S \setminus T$.)
The edge maps are the standard maps induced by the inclusion of generating sets.
    
\end{defn}

This is a simple complex of groups, and hence its fundamental group is given by the direct limit of the system of edge maps \cite[Def.~II.12.12]{bridson2013metric}. Clearly by our definitions, in this setting the fundamental group of $\widehat{\mathcal G}$ is $G_\Gamma$. 

By our discussion above, if $\Gamma$ is a straight line, this complex is developable with development $\widehat \Theta = \widehat \Theta_\Gamma$. 
If $\Gamma$ is a Coxeter diagram, then this definition coincides with the usual definition of the Coxeter complex of $W_\Gamma$, and hence is also developable. Thus $\widehat{\mathcal G}$ is developable for any finite Shephard group.
Sometimes when $\Gamma$ is a (non-extended) Coxeter diagram, we write $\widehat \Sigma(\Gamma) = \widehat \Sigma_\Gamma = \widehat \Theta_\Gamma$ and call $\widehat \Sigma_\Gamma$ the \emph{Coxeter complex} of $\Gamma$. It is well known that if $\Gamma$ is a Coxeter diagram, then $\widehat \Sigma_\Gamma$ possesses a natural metric under which it is isometric to a sphere. 
We discuss metrics on $\widehat \Theta_\Gamma$ for non-Coxeter $\Gamma$ in a later section.

Before concluding, we present an observation which will be of use later.

\begin{lemma} \label{lem:linkofcoxcomp}
    Let $\Gamma$ be a finite extended Coxeter diagram and $\delta$ a simplex of type $\widehat T$. Then
    \begin{align*}
        lk(\delta, \widehat \Theta) \cong \widehat \Theta(\Gamma(\widehat T)) = \widehat \Theta(\Gamma \setminus \Gamma(T)),
    \end{align*}
    where for a subdiagram $\Gamma'$ of $\Gamma$, we let $\Gamma \setminus \Gamma'$ denote the subdiagram of $\Gamma$ on vertex set $\mathrm{Vert}(\Gamma) \setminus \mathrm{Vert}(\Gamma')$.
\end{lemma}
\begin{proof}
    First consider a vertex $v$ of type $\hat s$.
    The vertex set of $lk(v, \widehat \Theta)$ is the collection of all vertices of $\widehat \Theta$ which are joined to $v$ by an edge, and incidence in the link is inherited from incidence in $\widehat \Theta$. 
    These vertices are the faces of $\mathcal P$ which are incident to the face $v$. Thus there are two types of faces, those containing $v$ and those contained by $v$. If a face contains $v$, then it's in the medial polytope between $v$ and $\mathbb{C}^d$ which forms one component of $\widehat\Theta(\Gamma(\hat s))$. If a face is contained in $v$, then it's in the medial polytope between $\varnothing$ and $v$, which forms the other component of $\widehat\Theta(\Gamma(\hat s))$. Thus $lk(v, \widehat \Theta) \cong \widehat \Theta(\Gamma(\widehat s))$.

    Now let $\delta$ be any simplex of type $\widehat T$. The link of $\delta$ is 
    \begin{align*}
        lk(\delta,\widehat \Theta) = \bigcap_{\substack{v \in \delta \\ v \text{ a vertex}}} lk(v,\widehat \Theta),
    \end{align*}
    with intersection interpreted on the level of posets.
    A vertex $v'$ of $\widehat \Theta$ is in each $lk(v,\widehat \Theta)$ if and only if it is a face in each $\widehat \Theta(\Gamma(\widehat s))$, $s \in T$, under the identification made in the previous paragraph. It is easy to see that this holds if and only if 
    \begin{align*}
        v' &\in \widehat \Theta\left( \bigcap_{s \in T}\Gamma(\widehat s) \right) \\
        & =\widehat \Theta\left( \Gamma\left(\bigcap_{s \in T}\widehat s\right) \right) \\
        &= \widehat \Theta(\Gamma(\widehat T)). \qedhere
    \end{align*}
\end{proof}

\subsection{A complex for arbitrary Shephard groups}

Throughout this section, let $\Gamma$ be any extended Coxeter diagram on vertex set $I$ with $S = \{s_i:i \in I\,\}$ 
the standard generators for $W = W_\Gamma$ and $G = G_\Gamma$.

The following definition is based on the definition of the ``modified Coxeter/Deligne complexes'' of \cite{charney1995k}.

\begin{defn}  \label{def:complexofshephard}
    Let $K = K_\Gamma = | (\mathcal S^{fs}_\Gamma )'|$, where $(\mathcal S^{fs}_\Gamma )'$ denotes the derived complex of $\mathcal S^{fs}_\Gamma$ and $|(\mathcal S^{fs}_\Gamma)'|$ is its geometric realization. 
    An $n$-simplex of $K$ is denoted
    \[
        [T_0 < T_1 < \dots < T_n]
    \] 
    for a chain $T_0 < T_1 < \dots < T_n$ with each $T_i \in \mathcal S^{fs}$. 
    In particular, the vertices are indexed by elements $T \in \mathcal S^{fs}$; we let $v_{T} = [T]$ denote the vertex of $K$ coming from $T$.
    We define a complex of groups $\mathcal G = \mathcal G(G_\Gamma,K_\Gamma)$ over $K$ by declaring the local group at $v_{T}$ to be $G_{\Gamma(T)}$ and the edge maps to be the natural maps coming from the inclusion of generators. 
\end{defn}

As with the complex for the finite Shephard groups, this is also a simple complex of groups. Hence by \cite[Def.~II.12.12]{bridson2013metric}, the fundamental group of $\mathcal G$ is the direct limit over the edge maps, and when $\mathcal S^f = \mathcal S^{fs}$, this is clearly $G_\Gamma$.

If $\Gamma$ is a Coxeter diagram, this complex of groups is developable \cite{charney1995k}, with development typically denoted $\Sigma = \Sigma_\Gamma$, called the \emph{modified Coxeter complex} or the \emph{Davis complex}. More generally, if $\mathcal G$ is developable, we will denote its development $\Theta = \Theta(\Gamma) = \Theta_\Gamma$ and may appropriately refer to it as the \emph{modified extended Coxeter complex}, \emph{modified Milnor fiber complex}, or \emph{extended Davis complex}.

A priori, it is not known in general if $\mathcal G$ is developable for an arbitrary extended Coxeter diagram. The rest of the paper is dedicated to studying developability and, more specifically, non-negative curvature of $\mathcal G$, utilizing the following (paraphrased) lemma.

\begin{lemma} \emph{\cite[Thm.~II.12.28]{bridson2013metric}}
    If $\mathcal H$ is a (simple) complex of groups over a simply connected domain and the local development at each vertex is locally $\mathrm{CAT}(0)$, then $\mathcal H$ is developable and has locally $\mathrm{CAT}(0)$ development.
\end{lemma}

It is clear that $K$ is simply connected ($[\varnothing]$ is a cone point). So in order to make use of this, we need a metric on $K$, which we will discuss shortly.
First, we describe a cell structure on $K$ which is coarser than the given simplicial structure.
The following definitions again are based on \cite{charney1995k}.

\begin{defn}
    For $T \in \mathcal S^{fs}$, let 
    \begin{align*}
        \mathcal S^{fs}_{\geq T} &= \{\,T' \in \mathcal S^{fs}:T' \supseteq T  \,\} &
        \mathcal S^{fs}_{> T}    &= \{\,T' \in \mathcal S^{fs}:T' \supsetneqq T\,\} \\
        \mathcal S^{fs}_{\leq T} &= \{\,T' \in \mathcal S^{fs}:T' \subseteq T  \,\} &
        \mathcal S^{fs}_{< T}    &= \{\,T' \in \mathcal S^{fs}:T' \subsetneqq T\,\} \\
        F_T &= |(\mathcal S^{fs}_{\geq T})'| &
        F_T^\circ &= |(\mathcal S^{fs}_{\geq T})'| \setminus |(\mathcal S^{fs}_{> T})'| \\
        F_T^* &= |(\mathcal S^{fs}_{\leq T})'| &
        (F_T^*)^\circ &= |(\mathcal S^{fs}_{\leq T})'| \setminus |(\mathcal S^{fs}_{< T})'|.
    \end{align*}
    We sometimes call $F_T$ ($F_T^\circ$)  a \emph{face} (resp.~\emph{open face}) of $K$ and $F_T^*$ ($(F_T^*)^\circ$) a \emph{dual face} (resp.~\emph{open dual face}) of $K$.
\end{defn}

Notice that $F_T^*$ is combinatorially a cube---the faces of the cube are of the form $F_{T_1} \cap F_{T_2}^*$ where $T_1 \subseteq T_2 \subseteq T$. So $K$ itself has a cubical structure where the cubical faces are $F_{T_1} \cap F_{T_2}^*$ for $T_1 \subseteq T_2 \in \mathcal S^{fs}$. The following is a standard exercise.

\begin{prop}
    With this cell structure, $F_T^*$ is isomorphic to the cone on the barycentric subdivision $\Delta_T'$ of $\Delta_T$ with cone point $v_T$. The isomorphism is given by sending a vertex $v_{T'}$ to the barycenter of $\sigma_{T \setminus T'}$ in $\Delta_T$.
\end{prop}

From here on, we identify $lk(v_T, F_T^*)$ and $\Delta_T$ in this way.
With this connection between $K$ and $\Delta_T$ for $T \in \mathcal S^{fs}$, we define the following metrics. 

\begin{defn}
    \begin{enumerate}
        \item (The cubical metric) Since the faces of $K$ are combinatorial cubes, metrize them to be standard Euclidean unit cubes. Under this metric, $\Delta_T$ is a spherical simplex with edge lengths all equal to $\pi/2$.

        \item (The Moussong metric) The definition of this metric is more involved and we refer to \cite[\S 4.4]{charney1995k} for details. 
        In summary, a cell $F_{T_1} \cap F_{T_2}^*$ is metrized to be a \emph{Coxeter block}, which is (the closure of) a connected component of the Coxeter zonotope associated to the finite Coxeter group $W_{T_2}$ minus its reflection hyperplanes. In this metric, if $T \in \mathcal S^{fs}$, then $\Delta_T$ is a spherical simplex where the length of the edge between the vertices corresponding to $s_i, s_j$ is $\pi - \pi_{m_{ij}}$. 
    \end{enumerate}
\end{defn}

We conclude by briefly recalling the definition of the local development as it applies here (our notation will differ slightly from \cite{bridson2013metric}).
In our case, the local development of $\mathcal G$ further relates $\mathcal G$ with $\widehat{\mathcal G}$.

Let $v_T$ be a vertex of $K$, with $T \in \mathcal S^{fs}$. The \emph{upper star} $St^T$ of $v_T$ in $\mathcal G$ is the (full) subcomplex of $K$ spanned by the vertices $v_{T'}$ with $T' \supseteq T$. 
The \emph{lower link} $Lk_T$ of $v_T$ in $\mathcal G$ is the development of the subcomplex of groups $\widehat {\mathcal G}(K_{<T})$ of $\mathcal G(K)$, where $K_{<T}$ denotes the subcomplex spanned by vertices $v_{T'}$ with $T' \subsetneq T$. 
Both of these objects are simplicial complexes which inherit the metric placed on $K$.
Then the \emph{local development at $v_T$} is (combinatorially) the join
\begin{align*}
    D(T) = St^T \ast Lk_T.
\end{align*}
Its metric naturally comes from the metric on $K$.
The link of $v_T$ in the local development is
\begin{align*}
    lk(v_T, D(T)) &= Lk^T \ast Lk_T,
\end{align*}
where $Lk^T$ is the \emph{upper link}, meaning the (full) subcomplex of $K$ spanned by the vertices $v_{T'}$ with $T' \supsetneq T$. We may also sometimes refer to this complex as $K_{>T}$.

Note that $K_{>T}$ is isomorphic to $lk(v_T, F_T)$ and $K_{<T}$ is isomorphic to $lk(v_T, F^*_T)$. 

We use the previous proposition to identify $K_{<T}$ with $\Delta_T$. With this identification, the complex of groups $\widehat{\mathcal G}(K_{<T})$ is isomorphic to $\widehat{\mathcal G}(G_T, \Delta_T)$ as defined before and thus $Lk_T$ is isomorphic to $\widehat \Theta_{\Gamma(T)}$. It is straightforward to check that the metrics placed on $K$ above agree with the claimed metrics on $\Delta_T$. This is summarized in the following

\begin{prop} \label{prop:linkdecomp}
    If $\Gamma$ is any extended Coxeter diagram, then in either metric, the link of a vertex $v_T$ in the local development of $\mathcal G(G_\Gamma, K_\Gamma)$ is isometric to the spherical join
    \[
        lk(v_T, F_T) * \widehat \Theta_{\Gamma(T)}.
    \]
\end{prop}

Showing the local development is nonpositively curved amounts to showing that these links are $\mathrm{CAT}(1)$. Since a spherical join is $\mathrm{CAT}(1)$ if and only if both components are \cite[Cor.~II.3.15]{bridson2013metric}, this reduces to showing that $lk(v_T, F_T)$ and $\widehat \Theta_{\Gamma(T)}$ are $\mathrm{CAT}(1)$ when $T \in \mathcal S^{fs}$.
We also have the following useful corollary for dealing with disconnected diagrams.

\begin{cor} \label{cor:ThetaSplits}
    If $\Gamma_1$, $\Gamma_2$ are extended Coxeter diagrams and $\Gamma = \Gamma_1 \sqcup \Gamma_2$ is their disjoint union, then 
    $\mathcal G(G_\Gamma) \cong \mathcal G(G_{\Gamma_1}) \times \mathcal G(G_{\Gamma_2})$.
     In particular, 
    \begin{enumerate}
        \item If $\mathcal G(G_\Gamma)$ is developable, then under either metric, $\Theta(\Gamma)$ is isometric to $\Theta(\Gamma_1) \times \Theta(\Gamma_2)$, and
        \item If $G_\Gamma$ is finite, then under either metric,
    $\widehat\Theta(\Gamma)$ is isometric to the spherical join $\widehat\Theta(\Gamma_1) * \widehat\Theta(\Gamma_2)$.
    \end{enumerate}
\end{cor}
\begin{proof}
    By a product $\mathcal G(G_1,K_1) \times \mathcal G(G_2,K_2)$ of (simple) complexes of groups, we mean the complex of groups over $K_1 \times K_2$ with vertex groups $G^{(1)}_{v_1} \times G^{(2)}_{v_2}$ for $v_i \in \mathrm{Vert}(K_i)$ and vertex group $G^{(i)}_{v_i}$ of $\mathcal G(G_i,K_i)$, and edge maps $\psi^{(1)}_{jk} \times \psi^{(2)}_{\ell n}$ with $\psi^{(i)}_{jk} : G^{(i)}_k \to G^{(i)}_j$ the edge maps of $\mathcal G(G_i,K_i)$.

    It is clear from the definitions that if $\Gamma = \Gamma_1 \sqcup \Gamma_2$ then $G_\Gamma \cong G_{\Gamma_1} \times G_{\Gamma_2}$ and $K_\Gamma \cong K_{\Gamma_1} \times K_{\Gamma_2}$. Thus we easily see that $\mathcal G(G_\Gamma) \cong \mathcal G(G_{\Gamma_1}) \times \mathcal G(G_{\Gamma_2})$. The statement about developability follows readily from this decomposition of the complexes of groups. 

    Now suppose $G_\Gamma$ is finite and $\Gamma = \Gamma_1 \sqcup \Gamma_2$. In particular, $\mathcal G(G_\Gamma)$ is developable since $G_\Gamma$ is finite. 
    Let $S$ be the standard generating set of $G_\Gamma$. 
    By our definitions, $K_{>S}$ is empty; hence $\widehat\Theta_\Gamma \cong K_{>S} * \widehat\Theta_\Gamma$.
    By Proposition \ref{prop:linkdecomp}, this is link of $v_S$ in the local development of $\mathcal G(G_\Gamma)$. But since the complex of groups is developable, this is simply $lk(v_S, \Theta_\Gamma)$.
    Since $\Theta(\Gamma) \cong \Theta(\Gamma_1) \times \Theta(\Gamma_2)$ by our previous result, 
    \begin{align*}
        \widehat\Theta_\Gamma
        &\cong lk(v_S, \Theta_\Gamma) \\
        &\cong lk(v_S, \Theta_{\Gamma_1} \times \Theta_{\Gamma_2}) \\
        &\cong lk(v_S, \Theta_{\Gamma_1}) * lk(v_S, \Theta_{\Gamma_2}) \\
        &\cong \widehat\Theta_{\Gamma_1} * \widehat\Theta_{\Gamma_2}. \qedhere
    \end{align*}
\end{proof}

\section{The cubical metric}
\label{sec:cubical}

In this section, we complete the proof of Theorem \ref{thm:cube}. Assume that $\mathcal S^{fs} = \mathcal S^f$ and the underlying Coxeter diagram of $\Gamma$ satisfies (FC). Endow $\mathcal G = \mathcal G(G_\Gamma, K_\Gamma)$ with the cubical metric.
As discussed at the end of the previous section, we must show that $lk(v_T, F_T)$ and $\widehat \Theta = \widehat \Theta_{\Gamma(T)}$ are $\mathrm{CAT}(1)$ when $T \in \mathcal S^{fs}$.

The cubical metric places a length of $\pi/2$ on each edge of $\widehat \Theta$. 
We apply Gromov's link condition, which says that a piecewise spherical simplicial complex with edge lengths all $\pi/2$ is $\mathrm{CAT}(1)$ if and only if it's a flag complex (see \cite[Thm.~II.5.18]{bridson2013metric} or \cite{gromov1987hyperbolic}).

Suppose $v_1 = [f_1] ,\dots,v_n = [f_n]$ are vertices of $\widehat \Theta$ which are pairwise connected by edges. This means each $f_i$ is a face of the polytope corresponding to $G_{\Gamma(T)}$, and the $f_i$ are pairwise nested. The only way a collection of affine subspaces of $\mathbb{C}^n$ can be pairwise nested is if they are linearly nested, resulting in a simplex $[f_1 < \dots < f_n]$ (up to change of indices) spanned by the vertices $v_i$. Hence $\widehat \Theta$ is flag, and therefore $\mathrm{CAT}(1)$ under the cubical metric. 

We now turn to $lk(v_T, F_T)$. As with $\widehat \Theta$, the cubical metric places a length of $\pi/2$ on each edge, so we must show that this is a flag complex. 
We can view $\mathcal S^{fs}_{> \varnothing}$ as an abstract simplicial complex with vertex set $S$, with a set $T$ spanning a simplex if and only if $W_{\Gamma(T)}$ is finite. Since the underlying Coxeter diagram of $\Gamma$ satisfies (FC), this is a flag complex. 
By the same argument in \cite[Lem.~4.3.3]{charney1995k}, we have that $lk(v_T,F_T) \cong lk(T, \mathcal S^{fs}_{> \varnothing})$. Since the link of a flag complex is still flag, and our previous remarks show that $\mathcal S^{fs}_{>\varnothing}$ is flag, it follows that $lk(v_T,F_T)$ is a flag complex, and hence $\mathrm{CAT}(1)$.

Therefore $\mathcal G$ is nonpositively curved, and hence developable with development $\Theta = \Theta_\Gamma$. As discussed before, $\Theta$ has the cell structure and metric of a cube complex. Since $\mathcal G$ is a simple complex of groups over a simply connected domain, we know that $\Theta$ is simply conneted and hence is a $\mathrm{CAT}(0)$ cube complex. Notice that $G_\Gamma$ acts cocompactly on $\Theta$ with fundamental domain $K$. The stabilizer of a cell of $K$ is the finite group $G_{\Gamma(T)}$ for some $T \in \mathcal S^{fs}$, implying the stabilizer of an arbitrary cell is a conjugate of this group, and hence finite. Thus $G$ also acts properly on $\Theta$, and therefore is cocompactly cubulated.

\section{A \texorpdfstring{$\mathrm{CAT}(1)$}{CAT(1)} criterion}
\label{sec:cat1crit}

Before proving Theorem \ref{thm:moussong}, we present a criterion for a simplicial complex made of $A_3$-simplices to be $\mathrm{CAT}(1)$. The idea behind this criteria comes from \cite{charney2004deligne} (which itself is based on \cite{elder2002curvature}), where it is used implicitly to show that the Moussong metric for the Deligne complex for the 4-strand braid group is $\mathrm{CAT}(0)$.

\begin{defn} \label{def:cccc}
    Let $\Delta$ be a spherical 2-simplex with vertices labeled $\hat a$, $\hat b$, $\hat c$ so that the angle at $\hat a$ and $\hat c$ is $\pi/3$ and the angle at $\hat b$ is $\pi/2$. For $g = a,b,c$, we label the edge opposite $\hat g$ as $g$. We call $\Delta$ the \emph{marked $A_3$ simplex}. 

    We call a homogeneous\footnote{Recall that an $n$-dimensional homogeneous simplicial complex is one where every simplex is contained in some $n$-simplex.} 2-dimensional simplicial complex $\Psi$ a \emph{marked $A_3$ simplicial complex} if every top dimensional simplex $\sigma$ is endowed with a choice of isomorphism $m_\sigma : \sigma \to \Delta$ so that if $x \in \sigma \cap \sigma'$, then $m_{\sigma}(x) = m_{\sigma'}(x)$.
    The maps $m_\sigma$ are the \emph{markings} of $\Psi$. 
    Sometimes we will call a top dimensional simplex a \emph{chamber}. 
    We call a vertex $v$ of $\Psi$ type $\hat a$, $\hat b$, or $\hat c$ if for some (hence any) chamber $\sigma$ containing $v$, the vertex $m_\sigma(v)$ of $\Delta$ is labeled $\hat a$, $\hat b$, $\hat c$, respectively. 
    Similarly, we call an edge $e$ type $a,b,c$ if $m_\sigma(e)$ is the edge of $\Delta$ labeled $a,b,c$, resp.
    Every marked $A_3$ simplicial complex has a metric, which we call the \emph{canonical metric}, obtained by pulling back the metric of $\Delta$ along the markings, so that each chamber is a simplex of shape $A_3$. 

    If $\Psi$ is a marked $A_3$ simplicial complex, we say it satisfies Charney's combinatorial $\mathrm{CAT}(1)$ criteria (CCCC) if
    \begin{enumerate}
        \item \label{item:largeaclink} The link of a vertex of type $\hat a$ or $\hat c$ has girth\footnote{The complex $\Psi$ is 2-dimensional, so the links of vertices are (simplicial) graphs. We often use the language of graph theory when referring to these links.} at least $6$,
        \item \label{item:developB} The link of a vertex of type $\hat b$ is a  complete bipartite graph which contains an embedded 4-cycle,
        \item \label{item:fillloop} Any edge path in Figure \ref{fig:shortclosededge} is contained in a corresponding subcomplex of $\Psi$ shown in Figure \ref{fig:fillededge}. (A shaded triangle represents a 2-simplex.)
    \end{enumerate}
\end{defn}

\begin{figure}[!ht]
    \centering
    \begin{tikzpicture}
        \matrix[matrix of nodes,column sep=-2pt,nodes={anchor=center, minimum height=0.1cm, align=flush center}]{
        \begin{tikzpicture}[scale=0.9]
        \node[left] at (-1.4,1) {(i)};

        \coordinate (A1) at (0,0) ;
        \coordinate (C1) at (1,1) ;
        \coordinate (C2) at (-1,1);
        \coordinate (A2) at (0,2) ;
        
        \filldraw (A1) circle (0.05cm);
        \filldraw (C1) circle (0.05cm);
        \filldraw (C2) circle (0.05cm);
        \filldraw (A2) circle (0.05cm);

        \node[below] at (A1) {$\hat a$};
        \node[right] at (C1) {$\hat c$};
        \node[left]  at (C2) {$\hat c$};
        \node[above] at (A2) {$\hat a$};

        \draw (A1) -- (C1) -- (A2) -- (C2) -- (A1);
        \end{tikzpicture}

        &
        \begin{tikzpicture}[scale=0.9]

        \node at (-2,1.1) {(ii.a)};

        \coordinate (B1) at (0,0.1)    ;
        \coordinate (A1) at (-1.4,0.1) ;
        \coordinate (A3) at (1.4,0.1)  ;
        \coordinate (A2) at (0,2.1)    ; 
        \coordinate (B2) at (-0.7, 1.1);
        \coordinate (B3) at (0.7, 1.1) ;

        \filldraw (B1) circle (0.05cm);
        \filldraw (A1) circle (0.05cm);
        \filldraw (A3) circle (0.05cm);
        \filldraw (A2) circle (0.05cm);
        \filldraw (B2) circle (0.05cm);
        \filldraw (B3) circle (0.05cm);

        \node[below]       at (B1) {$\hat b$};
        \node[below left]  at (A1) {$\hat a$};
        \node[below right] at (A3) {$\hat a$};
        \node[above]       at (A2) {$\hat a$}; 
        \node[above left]  at (B2) {$\hat b$};
        \node[above right] at (B3) {$\hat b$};

        \draw (B1) -- (A1) -- (B2) -- (A2) -- (B3) -- (A3) -- (B1); 
        \end{tikzpicture}
        &
        \begin{tikzpicture}[scale=0.9]

        \node at (-2,-2.1) {(ii.b)};

        \coordinate (B1b) at (0,-3.1)    ;
        \coordinate (A1b) at (-1.4,-3.1) ;
        \coordinate (A3b) at (1.4,-3.1)  ;
        \coordinate (A2b) at (0,-1.1)    ; 
        \coordinate (B2b) at (-0.7, -2.1);
        \coordinate (B3b) at (0.7, -2.1) ;

        \filldraw (B1b) circle (0.05cm);
        \filldraw (A1b) circle (0.05cm);
        \filldraw (A3b) circle (0.05cm);
        \filldraw (A2b) circle (0.05cm);
        \filldraw (B2b) circle (0.05cm);
        \filldraw (B3b) circle (0.05cm);

        \node[below]       at (B1b) {$\hat b$};
        \node[below left]  at (A1b) {$\hat c$};
        \node[below right] at (A3b) {$\hat c$};
        \node[above]       at (A2b) {$\hat c$}; 
        \node[above left]  at (B2b) {$\hat b$};
        \node[above right] at (B3b) {$\hat b$};

        \draw (B1b) -- (A1b) -- (B2b) -- (A2b) -- (B3b) -- (A3b) -- (B1b); 

        \end{tikzpicture}

        \\
        };
    \end{tikzpicture}
    \caption{The short closed edge paths}
    \label{fig:shortclosededge}
\end{figure}
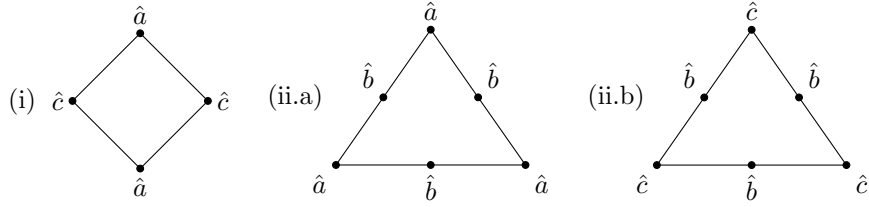

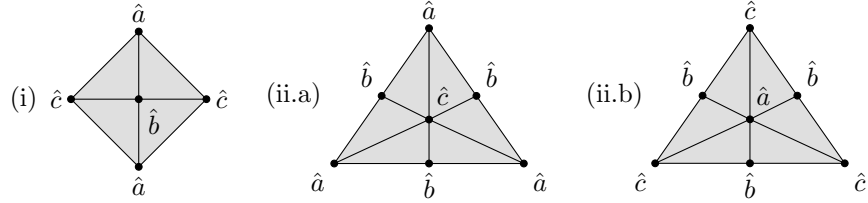
\begin{figure}[!ht]
    \centering
    \begin{tikzpicture}
        \matrix[matrix of nodes,column sep=-2pt,nodes={anchor=center, minimum height=0.1cm, align=flush center}]{
        
        \begin{tikzpicture}[scale=0.9]
        \node at (-1.7,-2.1) {(i)};
        \coordinate (A1) at (0,-3.1) ;
        \coordinate (C1) at (1,-2.1) ;
        \coordinate (C2) at (-1,-2.1);
        \coordinate (A2) at (0,-1.1) ;
        \coordinate (B1) at (0, -2.1);

        \draw[fill=gray!25] (A1) -- (C1) -- (A2) -- (C2) -- (A1);
        
        \filldraw (A1) circle (0.05cm);
        \filldraw (C1) circle (0.05cm);
        \filldraw (C2) circle (0.05cm);
        \filldraw (A2) circle (0.05cm);
        \filldraw (B1) circle (0.05cm);

        \node[below] at  (A1) {$\hat a$};
        \node[right] at  (C1) {$\hat c$};
        \node[left] at   (C2) {$\hat c$};
        \node[above] at  (A2) {$\hat a$};
        \node[below right] at (B1) {$\hat b$};

        \draw (B1) -- (A1);
        \draw (B1) -- (C1);
        \draw (B1) -- (C2);
        \draw (B1) -- (A2);
        \end{tikzpicture}
        &
        \begin{tikzpicture}[scale=0.9]
        \node at (-2,1.1) {(ii.a)};

        \coordinate (B1) at (0,0.1)    ;
        \coordinate (A1) at (-1.4,0.1) ;
        \coordinate (A3) at (1.4,0.1)  ;
        \coordinate (A2) at (0,2.1)    ; 
        \coordinate (B2) at (-0.7, 1.1);
        \coordinate (B3) at (0.7, 1.1) ;
        \coordinate (C1) at (0, 0.75)   ;

        \draw[fill=gray!25] (B1) -- (A1) -- (B2) -- (A2) -- (B3) -- (A3) -- (B1); 
        \draw (C1) -- (B1);
        \draw (C1) -- (A1);
        \draw (C1) -- (A3);
        \draw (C1) -- (A2);
        \draw (C1) -- (B2);
        \draw (C1) -- (B3);

        \filldraw (B1) circle (0.05cm);
        \filldraw (A1) circle (0.05cm);
        \filldraw (A3) circle (0.05cm);
        \filldraw (A2) circle (0.05cm);
        \filldraw (B2) circle (0.05cm);
        \filldraw (B3) circle (0.05cm);
        \filldraw (C1) circle (0.05cm);

        \node[below]       at (B1) {$\hat b$};
        \node[below left]  at (A1) {$\hat a$};
        \node[below right] at (A3) {$\hat a$};
        \node[above]       at (A2) {$\hat a$}; 
        \node[above left]  at (B2) {$\hat b$};
        \node[above right] at (B3) {$\hat b$};
        \node              at (0.2,1.1) {$\hat c$};

        \end{tikzpicture}
        &
        \begin{tikzpicture}[scale=0.9]
        \node at (-2,-2.1) {(ii.b)};

        \coordinate (B1b) at (0,-3.1)    ;
        \coordinate (A1b) at (-1.4,-3.1) ;
        \coordinate (A3b) at (1.4,-3.1)  ;
        \coordinate (A2b) at (0,-1.1)    ; 
        \coordinate (B2b) at (-0.7, -2.1);
        \coordinate (B3b) at (0.7, -2.1) ;
        \coordinate (C1b) at (0, -2.45)   ;

        \draw[fill=gray!25] (B1b) -- (A1b) -- (B2b) -- (A2b) -- (B3b) -- (A3b) -- (B1b); 
        \draw (C1b) -- (B1b);
        \draw (C1b) -- (A1b);
        \draw (C1b) -- (A3b);
        \draw (C1b) -- (A2b);
        \draw (C1b) -- (B2b);
        \draw (C1b) -- (B3b);

        \filldraw (B1b) circle (0.05cm);
        \filldraw (A1b) circle (0.05cm);
        \filldraw (A3b) circle (0.05cm);
        \filldraw (A2b) circle (0.05cm);
        \filldraw (B2b) circle (0.05cm);
        \filldraw (B3b) circle (0.05cm);
        \filldraw (C1b) circle (0.05cm);

        \node[below]       at (B1b) {$\hat b$};
        \node[below left]  at (A1b) {$\hat c$};
        \node[below right] at (A3b) {$\hat c$};
        \node[above]       at (A2b) {$\hat c$}; 
        \node[above left]  at (B2b) {$\hat b$};
        \node[above right] at (B3b) {$\hat b$};
        \node              at (0.2,-2.1) {$\hat a$};

        \end{tikzpicture}

        \\
        };
    \end{tikzpicture}
    \caption{Filling the short closed edge paths}
    \label{fig:fillededge}
\end{figure}

Our main theorem for this section is the following.

\begin{thm} \label{thm:charneycat1}
    If $\Psi$ is a marked $A_3$ simplicial complex which satisfies \tempname, then $\Psi$ is $\mathrm{CAT}(1)$ under its canonical metric.
\end{thm}

Before proving this, we need to establish some facts about the geometry of $\Psi$ under its canonical metric which follow from the above criteria. 

One of the main consequences (and inspirations) of our definition of a marked $A_3$ simplicial complex is that we may \emph{develop} certain geodesics from $\Psi$ onto the $A_3$ Coxeter complex $\widehat \Sigma = \widehat \Sigma(A_3)$. 
Let $W = W_{A_3}$ be the Coxeter group with diagram $A_3$ and generators $a,b,c$, i.e.,
\begin{align*}
    W = \langle\, a,b,c \mid a^2=b^2=c^2 = (ab)^3 = (bc)^3 = (ca)^2 = 1 \, \rangle.
\end{align*}
The vertices of $\widehat\Sigma$ are cosets of the form $wW_{\{a,b,c\} \setminus \{g\}}$, where $g \in \{a,b,c\}$. If we define $\hat g \coloneqq \{a,b,c\} \setminus \{g\}$, then the vertices are cosets of $W_{\hat g}$ for $g \in \{a,b,c\}$ and may sensibly be called \emph{type $\hat g$} (i.e., type $\hat a$, $\hat b$, or $\hat c$). Each chamber of $\widehat\Sigma$ is a simplex of type $A_3$, such that the angle of every simplex at a vertex of type ${\hat a}$ or ${\hat c}$ is $\pi/3$, and the angle of every simplex at a vertex of type $\hat b$ is $\pi/2$---hence the definition of the marked $A_3$ simplex $\Delta$ and its metric.

Suppose $\gamma$ is a local geodesic of $\Psi$ which does not intersect any vertices. Then the sequence of chambers which intersect $\gamma$ can be mapped down to a sequence of adjacent chambers of $\widehat \Sigma$ so that the markings of $\widehat \Sigma$ agree with the markings of $\Psi$. In particular, the image of $\gamma$ is a local geodesic in $\widehat \Sigma$, called the \emph{development of $\gamma$ (onto $\widehat \Sigma$)}.

Condition (\ref{item:developB}) allows us to develop a local geodesic which intersects vertices of type $\hat b$ in the following way. 
First we note that since $\Psi$ is marked, (\ref{item:developB}) implies that the link of a vertex of $\hat b$ is the join of a set of at least two vertices of type $\hat a$ with a set of at least two vertices of type $\hat c$. 
If $\gamma$ is a local geodesic which passes through a vertex $v$ of type $\hat b$, then it intersects the $\varepsilon$-sphere of $v$ at two points which are distance at least $\pi$ apart in the spherical metric on the $\varepsilon$-sphere. 
This sphere is isometric to the link of $v$, so the points on the sphere correspond to points in the link which are distance no less than $\pi$ apart.
Since the link is complete bipartite, they are exactly distance $\pi$ apart and are contained in a loop of edge length $4$ (since such a loop exists). 
This loop corresponds to 4 simplices of $\Psi$ which form a quadrilateral as in Figure \ref{fig:quadofB}. These simplices can then be mapped down to $\widehat \Sigma$, and the image of $\gamma$ will still be locally geodesic at the image of $v$.

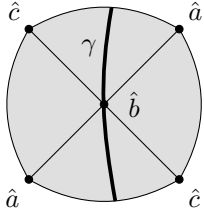
\begin{figure}[ht]
    \centering
    \begin{tikzpicture}
        \coordinate (A) at (0,0);
        \coordinate (B) at (2,0);
        \coordinate (C) at (0,2);
        \coordinate (D) at (2,2);
        \coordinate (E) at (1,1);

        \draw[fill=gray!25] (A) to[bend right] (B) to[bend right] (D) to[bend right] (C) to[bend right] (A);

        \draw (E) -- (A);
        \draw (E) -- (B);
        \draw (E) -- (C);
        \draw (E) -- (D);

        \filldraw (A) circle (0.05cm);
        \filldraw (B) circle (0.05cm);
        \filldraw (C) circle (0.05cm);
        \filldraw (D) circle (0.05cm);
        \filldraw (E) circle (0.05cm);

        \node[below left] at (A) {$\hat a$};
        \node[above right] at (D) {$\hat a$};
        \node[above left] at (C) {$\hat c$};
        \node[below right] at (B) {$\hat c$};

        \node[right=0.2cm] at (E) {$\hat b$};

        \coordinate (G1) at (1.1,2.285);
        \coordinate (G2) at (1.15,-0.29);
        \draw[line width=0.05cm] (G1) to[in=100,out=-100] (G2);

        \node at (0.8,1.75) {$\gamma$};
    \end{tikzpicture}
    \caption{Developing a geodesic through a vertex of type $\hat b$}
    \label{fig:quadofB}
\end{figure}

In fact, any local geodesic which passes through the (open) star of a vertex of type $\hat b$ is contained in a quadrilateral of this form. (If it doesn't pass through the vertex, this follows from the usual developing process, plus condition (\ref{item:developB}).) This leads us to the following

\begin{lemma} \label{lem:cutsortraverses}
    Suppose $v$ is a vertex of type $\hat b$ and $\gamma : [a_1,a_2] \to \mathrm{St}(v)$ is a geodesic segment with endpoints $v_i = \gamma(a_i)$ on the interior of an edge $e_i$ of $\partial \mathrm{St}(v)$. Then either $e_1$ and $e_2$ are adjacent (in which case we say $\gamma$ ``cuts a corner'' of $\mathrm{St}(v)$), or there is another edge $e_0$ adjacent to both $e_1$ and $e_2$ (in which case we say $\gamma$ ``traverses'' $\mathrm{St}(v)$).
\end{lemma}

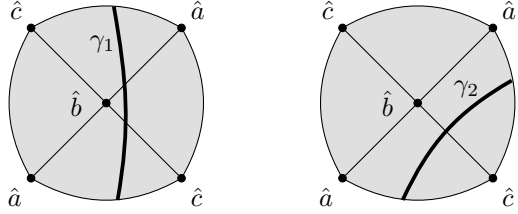
\begin{figure}[ht]
    \centering

    \begin{tikzpicture}
        \matrix[matrix of nodes,column sep=30pt]{
        \begin{tikzpicture}
        \coordinate (A) at (0,0);
        \coordinate (B) at (2,0);
        \coordinate (C) at (0,2);
        \coordinate (D) at (2,2);
        \coordinate (E) at (1,1);

        \draw[fill=gray!25] (A) to[bend right] (B) to[bend right] (D) to[bend right] (C) to[bend right] (A);

        \draw (E) -- (A);
        \draw (E) -- (B);
        \draw (E) -- (C);
        \draw (E) -- (D);

        \filldraw (A) circle (0.05cm);
        \filldraw (B) circle (0.05cm);
        \filldraw (C) circle (0.05cm);
        \filldraw (D) circle (0.05cm);
        \filldraw (E) circle (0.05cm);

        \node[below left] at (A) {$\hat a$};
        \node[above right] at (D) {$\hat a$};
        \node[above left] at (C) {$\hat c$};
        \node[below right] at (B) {$\hat c$};

        \node[left=0.2cm] at (E) {$\hat b$};

        \coordinate (G1) at (1.1,2.285);
        \coordinate (G2) at (1.15,-0.29);
        \draw[line width=0.05cm] (G1) to[in=80,out=-80] (G2);

        \node at (0.95,1.75) {$\gamma_1$};
    \end{tikzpicture}
    &
    \begin{tikzpicture}
        \coordinate (A) at (0,0);
        \coordinate (B) at (2,0);
        \coordinate (C) at (0,2);
        \coordinate (D) at (2,2);
        \coordinate (E) at (1,1);

        \draw[fill=gray!25] (A) to[bend right] (B) to[bend right] (D) to[bend right] (C) to[bend right] (A);

        \draw (E) -- (A);
        \draw (E) -- (B);
        \draw (E) -- (C);
        \draw (E) -- (D);

        \filldraw (A) circle (0.05cm);
        \filldraw (B) circle (0.05cm);
        \filldraw (C) circle (0.05cm);
        \filldraw (D) circle (0.05cm);
        \filldraw (E) circle (0.05cm);

        \node[below left] at (A) {$\hat a$};
        \node[above right] at (D) {$\hat a$};
        \node[above left] at (C) {$\hat c$};
        \node[below right] at (B) {$\hat c$};

        \node[left=0.2cm] at (E) {$\hat b$};

        \coordinate (G1) at (2.25,1.3);
        \coordinate (G2) at (0.8,-0.29);
        \draw[line width=0.05cm] (G1) to[in=65,out=210] (G2);

        \node at (1.65,1.15) {$\gamma_2$};
    \end{tikzpicture} 
    \\
    };
    \end{tikzpicture}
    \caption{Traversing ($\gamma_1$) and cutting a corner ($\gamma_2$) of a $\hat b$ star}
    \label{fig:traverseorcut}
\end{figure}

This follows easily from examining the aforementioned quadrilateral in the Coxeter complex. See Figure \ref{fig:traverseorcut}. The following can be seen in a similar way by developing to the Coxeter complex.

\begin{lemma} \label{lem:travorcut}
    Let $\gamma$ be a local geodesic of $\Psi$ avoiding vertices of type $\hat a$ and $\hat c$. If $\gamma$ traverses (resp, cuts a corner of) some star of a vertex of type $\hat b$, then it traverses (resp, cuts a corner of) every star of a vertex of type $\hat b$ whose interior intersects $\gamma$.
\end{lemma}

The following is one of the main lemmas of this section. It allows us to place a lower bound on the length of geodesics which avoid $\hat a$ and $\hat c$ vertices, as well as allows us to fill in other short loops.

\begin{lemma}
    \label{item:simpleintersection} The (closed) stars of two distinct vertices of type $\hat b$ intersect trivially, in exactly one vertex, or in exactly one edge.
\end{lemma}

\begin{figure}[!ht]
\vspace{-2em} \centering

\begin{tikzpicture}
\matrix[matrix of nodes,column sep=-2pt,nodes={anchor=center, minimum height=0.1cm, align=flush center}]{
\begin{tikzpicture}[scale=0.85]
\coordinate (A) at (0,0);
\coordinate (B) at (-1,1);
\coordinate (C) at (0,2);
\coordinate (D) at (1,1);
\filldraw (A) circle (0.05cm);
\filldraw (B) circle (0.05cm);
\filldraw (C) circle (0.05cm);
\filldraw (D) circle (0.05cm);
\draw (A) -- (B) -- (C) -- (D) -- (A);
\node[below] at (A) {$w_1$};
\node[left] at (B) {$v_1$};
\node[above] at (C) {$w_2$};
\node[right] at (D) {$v_2$};
\end{tikzpicture}
&
\begin{tikzpicture}[scale=0.9]
\coordinate (A) at (0,0);
\coordinate (B) at (-1.5,1);
\coordinate (C) at (0,2);
\coordinate (D) at (1.5,1);
\coordinate (E) at (0.75,1);
\coordinate (F) at (-0.75,1);
\node[below] at (A) {$w_1$};
\node[left] at (B) {$v_1$};
\node[above] at (C) {$w_2$};
\node[right] at (D) {$v_2$};
\node[right] at (F) {$u_1$};
\node[left] at (E) {$u_2$};
\draw[fill=gray!25] (E) -- (A) -- (D) -- (E);
\draw[fill=gray!25] (E) -- (C) -- (D) -- (E);
\draw[fill=gray!25] (F) -- (A) -- (B) -- (F);
\draw[fill=gray!25] (F) -- (C) -- (B) -- (F);
\filldraw (A) circle (0.05cm);
\filldraw (B) circle (0.05cm);
\filldraw (C) circle (0.05cm);
\filldraw (D) circle (0.05cm);
\filldraw (E) circle (0.05cm);
\filldraw (F) circle (0.05cm);
\end{tikzpicture}
&
\begin{tikzpicture}[scale=0.9]
\coordinate (A) at (0,0);
\coordinate (B) at (-2,1.5);
\coordinate (C) at (0,3);
\coordinate (D) at (2,1.5);
\coordinate (E) at (0.8,1.5);
\coordinate (F) at (-0.8,1.5);
\coordinate (G) at (0,1.5);
\draw[fill=gray!25] (E) -- (A) -- (D) -- (E);
\draw[fill=gray!25] (E) -- (C) -- (D) -- (E);
\draw[fill=gray!25] (F) -- (A) -- (B) -- (F);
\draw[fill=gray!25] (F) -- (C) -- (B) -- (F);
\draw[fill=gray!25] (C) -- (E) -- (A) -- (F) -- (C);
\draw (A) -- (C);
\draw (E) -- (F);
\node[below] at (A) {$w_1$};
\node[left] at (B) {$v_1$};
\node[above] at (C) {$w_2$};
\node[right] at (D) {$v_2$};
\node[above right] at (E) {$u_2$};
\node[above left] at (F) {$u_1$};
\node[below right] at (G) {$v$};
\filldraw (A) circle (0.05cm);
\filldraw (B) circle (0.05cm);
\filldraw (C) circle (0.05cm);
\filldraw (D) circle (0.05cm);
\filldraw (E) circle (0.05cm);
\filldraw (F) circle (0.05cm);
\filldraw (G) circle (0.05cm);
\end{tikzpicture} \\
};
\end{tikzpicture}
\vspace{-3.25em}

\caption{Filling a loop coming from the intersection of two stars}
\label{fig:starloop}
\end{figure}
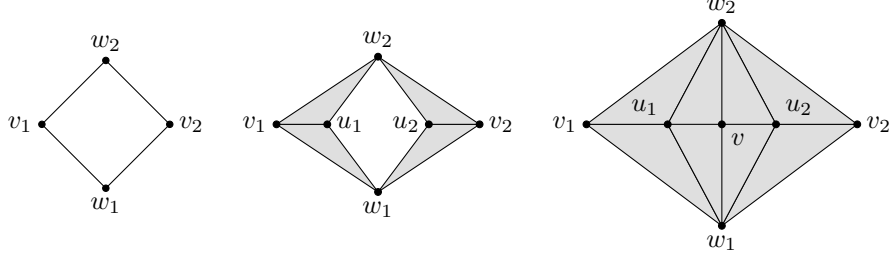

\begin{proof}
    Let $v_1, v_2$ be distinct vertices of type $\hat b$. We first show that $\mathrm{St}(v_1) \cap \mathrm{St}(v_2)$ contains at most one vertex of type $\hat a$ and at most one vertex of type $\hat c$.

    Let $w_1,w_2 \in \mathrm{St}(v_1) \cap \mathrm{St}(v_2)$ be vertices of type $\hat a$. This results in the leftmost edge path in Figure \ref{fig:starloop}.
    By condition (\ref{item:developB}), there are vertices $u_i \in \mathrm{St}(v_i)$ of type $\hat c$ so that $u_i$ is adjacent to $w_1,w_2$ for $i=1,2$. 
    This results in the middle complex of Figure \ref{fig:starloop}.
    The loop consisting of the $w_i$ and $u_i$ is a loop of type (i) (in Figure \ref{fig:shortclosededge}) and thus there is a vertex $v$ of type $\hat b$ so that the middle complex of Figure \ref{fig:starloop} can be augmented to the rightmost complex.

    If $v \neq v_1$ and $w_1 \neq w_2$, then the vertices $v_1, w_1,v,w_2$ give rise to an embedded loop of length $4$ in the link of $u_1$. However, condition (\ref{item:largeaclink}) prevents this from happening, so we must have either $v = v_1$ or $w_1 = w_2$. The same reasoning shows that either $v = v_2$ or $w_1 = w_2$. If $w_1 \neq w_2$, we would have $v_1 = v = v_2$, contradicting the assumption that $v_1 \neq v_2$. So, we must have $w_1 = w_2$ and thus $\mathrm{St}(v_1) \cap \mathrm{St}(v_2)$ contains at most one vertex of type $\hat a$. An analagous argument shows that it contains at most one vertex of type $\hat c$.

    Now suppose $w,u \in \mathrm{St}(v_1) \cap \mathrm{St}(v_2)$ are vertices of type $\hat a$ and $\hat c$, respectively. As noted before, the definition of marked $A_3$ simplicial complex and condition (\ref{item:developB}) imply that the link of a vertex of type $\hat b$ is a join of a set of vertices of type $\hat a$ with a set of vertices of type $\hat c$. This means there is an edge, which is necessarily unique, between $w$ and $u$, and this edge  contained in $\mathrm{St}(v_1) \cap \mathrm{St}(v_2)$. 

    The result follows from noticing that if $e_1,e_2$ are two distinct edges, then their union contains at least three vertices (since $\Psi$ is simplicial), and thus $\mathrm{St}(v_1) \cap \mathrm{St}(v_2)$ cannot contain each of these vertices by our above work, meaning at least one of the edges is not contained in this intersection.
\end{proof}

As a first consequence, we have

\begin{lemma} \label{lem:completinglastpath}
    Any edge path appearing in Figure \ref{fig:anotherpath} is contained in a subcomplex of $\Psi$ found in Figure \ref{fig:anotherpathfilled}.
\end{lemma}
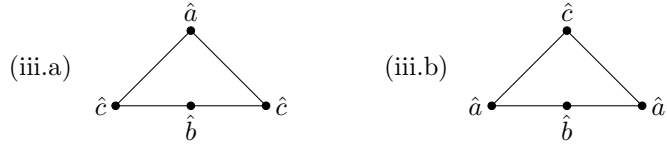
\begin{figure}[ht]
    \begin{tikzpicture}
        \coordinate (A1) at (0,1) ;
        \coordinate (C1) at (1,1) ;
        \coordinate (C2) at (-1,1);
        \coordinate (A2) at (0,2) ;
        
        \filldraw (A1) circle (0.05cm);
        \filldraw (C1) circle (0.05cm);
        \filldraw (C2) circle (0.05cm);
        \filldraw (A2) circle (0.05cm);

        \node[below] at (A1) {$\hat b$};
        \node[right] at (C1) {$\hat c$};
        \node[left]  at (C2) {$\hat c$};
        \node[above] at (A2) {$\hat a$};

        \draw (A1) -- (C1) -- (A2) -- (C2) -- (A1);

        \coordinate (A1b) at  at (5,1) ;
        \coordinate (C1b) at  at (6,1) ;
        \coordinate (C2b) at  at (4,1);
        \coordinate (A2b) at  at (5,2) ;
        
        \filldraw (A1b) circle (0.05cm);
        \filldraw (C1b) circle (0.05cm);
        \filldraw (C2b) circle (0.05cm);
        \filldraw (A2b) circle (0.05cm);

        \node[below] at (A1b) {$\hat b$};
        \node[right] at (C1b) {$\hat a$};
        \node[left]  at (C2b) {$\hat a$};
        \node[above] at (A2b) {$\hat c$};

        \draw (A1b) -- (C1b) -- (A2b) -- (C2b) -- (A1b);

        \node at (-2,1.5) {(iii.a)};
        \node at (3,1.5) {(iii.b)};
    \end{tikzpicture}
    \caption{Two more short closed edge paths}
    \label{fig:anotherpath}
\end{figure}
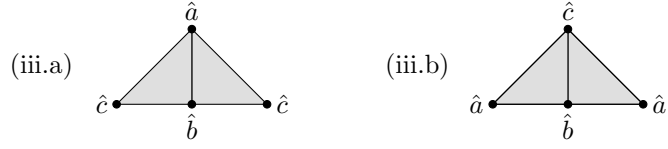
\begin{figure}[ht]
    \begin{tikzpicture}
        \coordinate (A1) at (0,1) ;
        \coordinate (C1) at (1,1) ;
        \coordinate (C2) at (-1,1);
        \coordinate (A2) at (0,2) ;
        \draw[fill=gray!25] (A1) -- (C1) --(A2) -- (A1);
        \draw[fill=gray!25] (C2) -- (A2) --(A1) -- (C2);
        
        \filldraw (A1) circle (0.05cm);
        \filldraw (C1) circle (0.05cm);
        \filldraw (C2) circle (0.05cm);
        \filldraw (A2) circle (0.05cm);

        \node[below] at (A1) {$\hat b$};
        \node[right] at (C1) {$\hat c$};
        \node[left]  at (C2) {$\hat c$};
        \node[above] at (A2) {$\hat a$};

        \coordinate (A1b) at  at (5,1) ;
        \coordinate (C1b) at  at (6,1) ;
        \coordinate (C2b) at  at (4,1);
        \coordinate (A2b) at  at (5,2) ;
        \draw[fill=gray!25] (A1b) -- (C1b) --(A2b) -- (A1b);
        \draw[fill=gray!25] (C2b) -- (A2b) --(A1b) -- (C2b);
        
        \filldraw (A1b) circle (0.05cm);
        \filldraw (C1b) circle (0.05cm);
        \filldraw (C2b) circle (0.05cm);
        \filldraw (A2b) circle (0.05cm);

        \node[below] at (A1b) {$\hat b$};
        \node[right] at (C1b) {$\hat a$};
        \node[left]  at (C2b) {$\hat a$};
        \node[above] at (A2b) {$\hat c$};

        \draw (A1b) -- (C1b) -- (A2b) -- (C2b) -- (A1b);

        \node at (-2,1.5) {(iii.a)};
        \node at (3,1.5) {(iii.b)};
    \end{tikzpicture}
    \caption{Filling the edge paths}
    \label{fig:anotherpathfilled}
\end{figure}
\begin{proof}
    We will deal with the left loop of Figure \ref{fig:anotherpath}; the argument for the other loop is identical. By condition (\ref{item:developB}), we know that this loop is a subcomplex (indicated in bold) of the leftmost complex in Figure \ref{fig:fillingpart1}, since two $\hat c$ vertices in a star of a $\hat b$ vertex must be connected by an $\hat a$ vertex (which is also in this star).
    \begin{figure}[!ht]
    \begin{tikzpicture}
    \matrix[matrix of nodes,column sep=12pt,nodes={anchor=center, minimum height=0.1cm, align=flush center}]{
    \begin{tikzpicture}
    \coordinate (A1) at (0,1) ;
    \coordinate (C1) at (1,1) ;
    \coordinate (C2) at (-1,1);
    \coordinate (A2) at (0,2) ;
    \coordinate (A) at (0,0)  ;

    \draw[fill=gray!25] (C1) -- (A1) -- (A) -- (C1);
    \draw[fill=gray!25] (C2) -- (A1) -- (A) -- (C2);
    
    \filldraw (A1) circle (0.05cm);
    \filldraw (C1) circle (0.05cm);
    \filldraw (C2) circle (0.05cm);
    \filldraw (A2) circle (0.05cm);
    \filldraw (A) circle (0.05cm);

    \node[above] at (A1) {$\hat b$};
    \node[right] at (C1) {$\hat c$};
    \node[left]  at (C2) {$\hat c$};
    \node[above] at (A2) {$\hat a$};
    \node[below] at (A)  {$\hat a$};

    \draw[line width=0.04cm]  (A1) -- (C1) -- (A2) -- (C2) -- (A1);
    \end{tikzpicture}
    &
    \begin{tikzpicture}
    \coordinate (A1) at (0,1) ;
    \coordinate (C1) at (1,1) ;
    \coordinate (C2) at (-1,1);
    \coordinate (A2) at (0,2) ;
    \coordinate (A) at (0,0)  ;
    \draw (A1) -- (A2);
    \draw[fill=gray!25] (A2) -- (C1) -- (A) -- (C2) -- (A2);
    \draw (A2) -- (A);
    
    \filldraw (A1) circle (0.05cm);
    \filldraw (C1) circle (0.05cm);
    \filldraw (C2) circle (0.05cm);
    \filldraw (A2) circle (0.05cm);
    \filldraw (A) circle (0.05cm);

    \node[above right] at (A1) {$\hat b$};
    \node[right] at (C1) {$\hat c$};
    \node[left]  at (C2) {$\hat c$};
    \node[above] at (A2) {$\hat a$};
    \node[below] at (A)  {$\hat a$};

    \draw[line width=0.04cm]  (A1) -- (C1) -- (A2) -- (C2) -- (A1);
    \end{tikzpicture}
    \\};
    \end{tikzpicture}
    \caption{Completing the loop}
    \label{fig:fillingpart1}
    \end{figure}
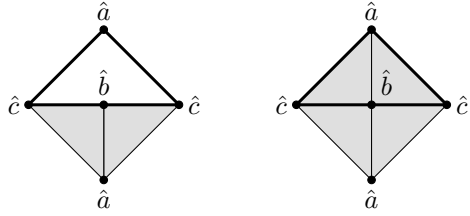
    But now the vertices of type $\hat a$ and $\hat c$ form a loop of type (i) (Figure \ref{fig:shortclosededge}), and in particular are contained in the star of a common $\hat b$ vertex. If this vertex were distinct from the one already included in our original loop, this would contradict Lemma \ref{item:simpleintersection}. So the complex on the left of Figure \ref{fig:fillingpart1} can be completed to the complex on the right, giving the desired result.
\end{proof}

We now show that the loops considered so far are the only possible short loops that can arise.

\begin{lemma} \label{lem:onlyshortloops}
    Suppose $p$ is a simple closed edge path of $\Psi$. Then $\ell(p) < 2\pi$ (where $\ell$ is the length function) if and only if $p$ is one of the paths of type (i), (ii), or (iii) in Figures \ref{fig:shortclosededge} or \ref{fig:anotherpath}.
\end{lemma}
\begin{proof}
    An edge of $p$ from a vertex of type $\hat a$ (or $\hat c$) to $\hat b$ is immediately followed by an edge from a vertex of type $\hat b$ to $\hat a$ (resp.~$\hat c$). This is a consequence of condition (\ref{item:developB}) and the fact that the length of edges of the link is $\pi/2$. Moreover, as a consequence of the definition of a marked $A_3$ simplicial complex, vertices of type $\hat a$ or $\hat c$ cannot be joined by an edge to a vertex of the same type. Thus $p$ consists of edges between $\hat a$ vertices and $\hat c$ vertices, the length of which is $\beta \coloneqq \arccos(1/\sqrt{3})$, and pairs of edges from $\hat a$ to $\hat b$ to $\hat a$, or from $\hat c$ to $\hat b$ to $\hat c$, the total length of which is $2\alpha$, where $\alpha \coloneqq \arccos(1/3)$. We conclude by an argument completely identical to that found in \cite[Pf. of Thm. 5.4]{charney2004deligne} that the only possible edge paths of length less than $2\pi$ are those in Figures \ref{fig:shortclosededge} or \ref{fig:anotherpath} (which are those of Figure 7 of \cite{charney2004deligne}), as desired.
\end{proof}

We are now ready to prove the main theorem of the section. The argument is completely analagous to that of \cite[Pf. of Thm. 5.4]{charney2004deligne}. We give the argument here to clarify that it works in this more general setting.

\begin{proof}[Proof (of Theorem \ref{thm:charneycat1})]
    First we note that the links of vertices are $\mathrm{CAT}(1)$. For a vertex $v$ of type $\hat a$ or $\hat c$, the canonical metric places a length of $\pi/3$ on the edges of $lk(v,\Psi)$. Condition (\ref{item:largeaclink}) guarantees the (combinatorial) girth of this link is 6, and hence the length of any nontrivial closed loop in $lk(v,\Psi)$ is at least $6\pi/3 = 2\pi$. If $v$ is a vertex of type $\hat b$, then the canonical metric places a length of $\pi/2$ on the edges of $lk(v,\Psi)$. Since this link is bipartite, it has girth at least $4$, so the length of any nontrivial closed loop in $lk(v,\Psi)$ is at least $4\pi/2 = 2\pi$. Hence the vertex links are $\mathrm{CAT}(1)$. It remains to show that if $\gamma$ is a closed geodesic of $\Psi$, then $\ell(\gamma) \geq 2\pi$. 
    So, suppose $\gamma$ is a closed geodesic. There are three cases to consider.

    First assume $\gamma$ is an edge path. If $\ell(\gamma) < 2\pi$, then Lemma \ref{lem:onlyshortloops} implies $\gamma$ is one of the paths of type (i), (ii), or (iii) (Figures \ref{fig:shortclosededge} and \ref{fig:anotherpath}). Condition (\ref{item:fillloop}) (resp.~Lemma \ref{lem:completinglastpath}) guarantees that the paths of type (i) and (ii) (resp.~(iii)) are not locally geodesic at any vertex of type $\hat a$ or $\hat c$, and thus we must have $\ell(\gamma) \geq 2\pi$. 

    Now assume $\gamma$ intersects the interior of at least one 2-cell and does not intersect any vertices of type $\hat a$ or $\hat c$. 
    Let $v_1,\dots,v_n$ denote the distinct vertices of type $\hat b$ such that $\gamma \cap st(v_i) \neq \varnothing$.
    Since such a star is $\mathrm{CAT}(1)$ and has diameter $< \pi$, we know $n \neq 1$.
    By Lemmas \ref{lem:cutsortraverses} and \ref{item:simpleintersection}, $\gamma$ cannot close up after intersecting only two such stars, so $n \neq 2$.
    By Lemma \ref{lem:travorcut}, there are two cases now to consider; either $\gamma$ traverses every $\mathrm{St}(v_i)$  or cuts a corner of every $\mathrm{St}(v_i)$.

    \begin{figure}[ht]
        \centering
        \begin{tikzpicture}[scale=0.8]
        \coordinate (A1) at (0,0);
        \coordinate (C1) at (2,0);
        \coordinate (A2) at (4,0);
        \coordinate (C2) at (6,0);
        \coordinate (A3) at (6,2);
        \coordinate (C3) at (4,2);
        \coordinate (A4) at (2,2);
        \coordinate (C4) at (0,2);
        \coordinate (v1) at (1,1);
        \coordinate (v2) at (3,1);
        \coordinate (v3) at (5,1);
        \draw[fill=gray!25] (A1) --
            (C1) --
            (A2) --
            (C2) --
            (A3) --
            (C3) --
            (A4) --
            (C4) -- (A1);
        \draw (C1) -- (A4);
        \draw (C3) -- (A2);
        \draw (C4) -- (C1) -- (C3) -- (C2);
        \draw (A1) -- (A4) -- (A2) -- (A3);

        \usetikzlibrary{decorations.markings}

        \begin{scope}[decoration={
            markings,
            mark=at position 0.5 with {\arrow[line width=0.05cm]{stealth}}}
            ] 
            \draw[postaction={decorate}] (A1) -- (C4);
            \draw[postaction={decorate}] (A3) -- (C2);
        \end{scope}

        \filldraw (A1) circle (0.05cm);
        \filldraw (C1) circle (0.05cm);
        \filldraw (A2) circle (0.05cm);
        \filldraw (C2) circle (0.05cm);
        \filldraw (A3) circle (0.05cm);
        \filldraw (C3) circle (0.05cm);
        \filldraw (A4) circle (0.05cm);
        \filldraw (C4) circle (0.05cm);
        \filldraw (v1) circle (0.05cm);
        \filldraw (v2) circle (0.05cm);
        \filldraw (v3) circle (0.05cm);

        \node[below] at (A1) {$\hat a$};
        \node[below] at (A2) {$\hat a$};
        \node[above] at (A3) {$\hat a$};
        \node[above] at (A4) {$\hat a$};
        \node[below] at (C1) {$\hat c$};
        \node[below] at (C2) {$\hat c$};
        \node[above] at (C3) {$\hat c$};
        \node[above] at (C4) {$\hat c$};

        \node[right] at (v1) {$v_1$};
        \node[right] at (v2) {$v_2$};
        \node[right] at (v3) {$v_3$};

        \draw[dashed] (0,0.45) -- (6,1.75);

        \draw[line width=0.04cm] (C2) -- (A3) -- (C4);

        \end{tikzpicture}
        \caption{The case $n = 3$}
        \label{fig:threestars}
    \end{figure}
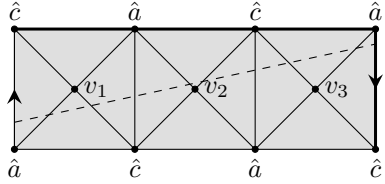
    Suppose first that $\gamma$ traverses these stars. We claim that $n \neq 3$. 
    To the contrary, if $n = 3$ then we have a subcomplex of $\Psi$ seen in Figure \ref{fig:threestars}, with $\gamma$ indicated by a dashed line and the edges marked with arrows identified. This results in an edge path in the union of the $\mathrm{St}(v_i)$ which forms a path of type (i) above but is not contained in any of the $\mathrm{St}(v_i)$, indicated in bold in Figure \ref{fig:threestars}.
    By condition (\ref{item:fillloop}), it follows that this path is contained in some $\mathrm{St}(v')$ for a vertex $v'$ of type $\hat b$ distinct from each $v_i$. 
    But then $\mathrm{St}(v')$ would intersect $\mathrm{St}(v_3)$ in more than one edge, contradicting Lemma \ref{item:simpleintersection}. 
    By developing onto $\widehat \Sigma$, any geodesic which traverses 4 quadrilaterals of the type in Figure \ref{fig:quadofB} has length no less than $2\pi$.

Now suppose $\gamma$ cuts corners. 
We claim $n \geq 6$, resulting in a complex of the form in Figure \ref{fig:sixstars}, with $\gamma$ represented by a dashed line.
\begin{figure}[ht]
    \begin{tikzpicture}[scale=0.8]
        \coordinate (A1) at (0,0);
        \coordinate (A2) at (2,2);
        \coordinate (A3) at (4,4);
        \coordinate (A4) at (6,6);
        \coordinate (C1) at (0,2);
        \coordinate (C2) at (2,4);
        \coordinate (C3) at (4,6);
        \coordinate (C4) at (8,6);

        \coordinate (A5) at (8,4);
        \coordinate (C5) at (6,4);
        \coordinate (A6) at (6,2);
        \coordinate (C6) at (4,2);
        \coordinate (A7) at (4,0);
        \coordinate (C7) at (2,0);

        \coordinate (v1) at (1,1);
        \coordinate (v2) at (3,1);
        \coordinate (v3) at (3,3);
        \coordinate (v4) at (5,3);
        \coordinate (v5) at (5,5);
        \coordinate (v6) at (7,5);

        \draw[fill=gray!25]
        (A1) -- (C1) --
        (A2) -- (C2) --
        (A3) -- (C3) --
        (A4) -- (C4) --
        (A5) -- (C5) --
        (A6) -- (C6) --
        (A7) -- (C7) -- (A1);

        \draw (A2) -- (C7);
        \draw (A2) -- (C6);
        \draw (A3) -- (C6);
        \draw (A3) -- (C5);
        \draw (A4) -- (C5);

        \draw (C1) -- (C7);
        \draw (C2) -- (C6);
        \draw (C3) -- (C5);
        \draw (A2) -- (A7);
        \draw (A3) -- (A6);
        \draw (A4) -- (A5);

        \draw (A1) -- (A2);
        \draw (A3) -- (A2);
        \draw (A3) -- (A4);
        \draw (C7) -- (C6);
        \draw (C5) -- (C6);
        \draw (C5) -- (C4);

        \filldraw (A1) circle (0.05cm);
        \filldraw (A2) circle (0.05cm);
        \filldraw (A3) circle (0.05cm);
        \filldraw (A4) circle (0.05cm);
        \filldraw (A5) circle (0.05cm);
        \filldraw (A6) circle (0.05cm);
        \filldraw (A7) circle (0.05cm);
        \filldraw (C1) circle (0.05cm);
        \filldraw (C2) circle (0.05cm);
        \filldraw (C3) circle (0.05cm);
        \filldraw (C4) circle (0.05cm);
        \filldraw (C5) circle (0.05cm);
        \filldraw (C6) circle (0.05cm);
        \filldraw (C7) circle (0.05cm);

        \filldraw (v1) circle (0.05cm);
        \filldraw (v2) circle (0.05cm);
        \filldraw (v3) circle (0.05cm);
        \filldraw (v4) circle (0.05cm);
        \filldraw (v5) circle (0.05cm);
        \filldraw (v6) circle (0.05cm);

        \node[left=0.05cm] at (v1) {$v_1$};
        \node[right=0.05cm] at (v2) {$v_2$};
        \node[left=0.05cm] at (v3) {$v_3$};
        \node[right=0.05cm] at (v4) {$v_4$};
        \node[left=0.05cm] at (v5) {$v_5$};
        \node[right=0.05cm] at (v6) {$v_6$};

        \draw[dashed] (0.4,0) -- (7.5,6);

        \draw (C5) circle (0.15cm);
        \draw (C6) circle (0.15cm);
        \draw (C7) circle (0.15cm);
    \end{tikzpicture}
    \caption{A geodesic cutting corners}
    \label{fig:sixstars}
\end{figure}
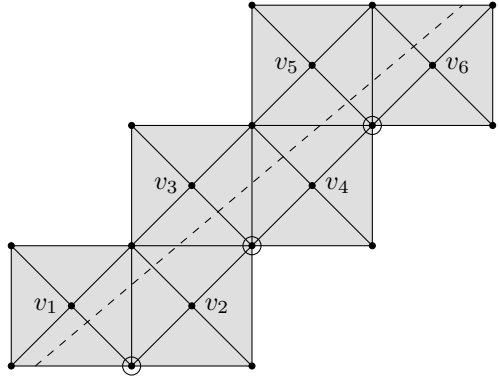
If $\gamma$ closes up after cutting less than 6 corners, then two of the three circled vertices in Figure \ref{fig:sixstars} would be identified, which cannot happen (either because $\Psi$ is simplicial and each simplex in the diagram is distinct, or because of Lemma \ref{item:simpleintersection}). Hence we must have $n \geq 6$. By developing, a geodesic which cuts at least 6 corners has length $\geq 2\pi$.

    Last, assume $\gamma$ intersects at least one vertex of type $\hat a$ or $\hat c$ but is not an edge path. Decompose $\gamma$ into the concatination of segments $\gamma = \gamma_0\gamma_1\dots\gamma_n$ with no vertices of type $\hat a$ or $\hat c$ in the interior of the $\gamma_i$. 
    Note that at least one of the $\gamma_i$ contains no edges of $\Psi$. 
    By developing we see that any such $\gamma_i$ is half a great circle from a vertex of type $\hat a$ or $\hat c$ to a vertex of type $\hat c$ or $\hat a$, respectively,  
    so the length of $\gamma_i$ is $\pi$---see Figure \ref{fig:nonedgegeo}.
    Thus if there are two segments of this type, $\ell(\gamma) \geq 2\pi$, so suppose there is exactly one segment $\gamma_i$ which does not contain an edge of $\Psi$.
    We may assume $\gamma_0$ is this segment. 
    Notice in particular that since the endpoints of $\gamma_0$ always have different type, $\gamma_0$ cannot be closed, or in other words $n > 0$.

    \begin{figure}[ht]
        \centering

        \begin{tikzpicture}
            \coordinate (A1) at (0.2,0);
            \coordinate (A2) at (5.25,0);
            \coordinate (B1) at (2,-1);
            \coordinate (C1) at (4,-0.9);
            \coordinate (C2) at (1.5,0.8);
            \coordinate (B2) at (3.5,1);

            \draw[fill=gray!25] (A1) to[out=-50, in=170] (B1) to[out=-5,in=195] (C1) to[out=20,in=230] (A2) to[out=130,in=-10] (B2) to[out=175,in=15] (C2) to[out=200,in=40] (A1);

            \filldraw (A1) circle (0.05cm);
            \filldraw (A2) circle (0.05cm);
            \filldraw (B1) circle (0.05cm);
            \filldraw (C1) circle (0.05cm);
            \filldraw (B2) circle (0.05cm);
            \filldraw (C2) circle (0.05cm);

            \draw (C2) to[out=-85,in=115] (B1) -- (B2) to[out=-65,in=95] (C1);

            \node[left] at (A1) {$\hat a$};
            \node[right] at (A2) {$\hat c$};
            \node[below] at (B1) {$\hat c$};
            \node[below] at (C1) {$\hat b$};
            \node[above] at (C2) {$\hat b$};
            \node[above] at (B2) {$\hat a$};

            \draw[line width=0.065cm] (A1) to[out=15,in=165] (A2);

            \node at (4.4,0) {$\gamma_0$};
        \end{tikzpicture}
        \caption{A geodesic segment $\gamma_0$ in $\Psi$ which is not an edge path}
        \label{fig:nonedgegeo}
    \end{figure}
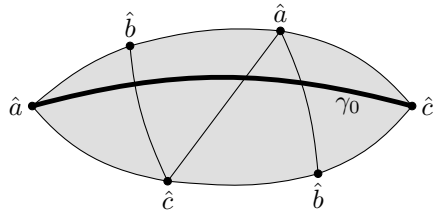

    Let $\overline \gamma_0$ be the image of $\gamma_0$ under the development to $\widehat \Sigma$. 
    Rotate $\overline \gamma_0$ within $\widehat \Sigma$ relative to its endpoints until it becomes an edge path with vertices of type $\hat a$ or $\hat c$ in its interior. 
    Then lift this rotation to $\gamma_0'$ in $\Psi$ so that its endpoints agree with the original endpoints of $\gamma_0$ (see Figure \ref{fig:rotatedgeo}). This new path is clearly locally geodesic in its interior and has the same length as $\gamma_0$. Let $\gamma' = \gamma_0'\gamma_1\dots \gamma_n$. Note that $\ell(\gamma') = \ell(\gamma)$. We claim that $\gamma'$ is locally geodesic. 
    To do this, we only need to check the endpoints of $\gamma_0'$.

    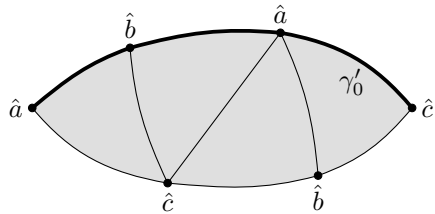
\begin{figure}[ht]
        \centering

        \begin{tikzpicture}
            \coordinate (A1) at (0.2,0);
            \coordinate (A2) at (5.25,0);
            \coordinate (B1) at (2,-1);
            \coordinate (C1) at (4,-0.9);
            \coordinate (C2) at (1.5,0.8);
            \coordinate (B2) at (3.5,1);

            \draw[fill=gray!25] (A1) to[out=-50, in=170] (B1) to[out=-5,in=195] (C1) to[out=20,in=230] (A2) to[out=130,in=-10] (B2) to[out=175,in=15] (C2) to[out=200,in=40] (A1);

            \draw[line width=0.05cm] (A2) to[out=130,in=-10] (B2) to[out=175,in=15] (C2) to[out=200,in=40] (A1);

            \filldraw (A1) circle (0.05cm);
            \filldraw (A2) circle (0.05cm);
            \filldraw (B1) circle (0.05cm);
            \filldraw (C1) circle (0.05cm);
            \filldraw (B2) circle (0.05cm);
            \filldraw (C2) circle (0.05cm);

            \draw (C2) to[out=-85,in=115] (B1) -- (B2) to[out=-65,in=95] (C1);

            \node[left] at (A1) {$\hat a$};
            \node[right] at (A2) {$\hat c$};
            \node[below] at (B1) {$\hat c$};
            \node[below] at (C1) {$\hat b$};
            \node[above] at (C2) {$\hat b$};
            \node[above] at (B2) {$\hat a$};

            \node at (4.45,0.35) {$\gamma_0'$};
        \end{tikzpicture}
        \caption{A possible rotation of $\gamma_0$ in $\Psi$}
        \label{fig:rotatedgeo}
    \end{figure}

    Since $\gamma_0$ is the unique non-edge path in the decomposition of $\gamma$,
    we know that $\gamma_1$ and $\gamma_n$ are both edge paths in $\Psi$. (The case $n = 1$ is acceptable.) 
    Let $v_0$ be the vertex between $\gamma_0$ and $\gamma_1$, and let $v_n$ be the vertex between $\gamma_n$ and $\gamma_0$ (if $n = 1$, let these vertices be the distinct endpoints of $\gamma_0$). 
    Consider $lk(v_j,\Psi)$ for $j=1$ and $j=n$. 
    The intersection of the $\varepsilon$-sphere of $v_j$ with $\gamma_0$ and $\gamma_j$ gives points in $lk(v_j,\Psi)$ which are distance at least $\pi$ apart. 
    Since $\gamma_j$ is an edge path, it gives a vertex $v'_j$ of $lk(v_j,\Psi)$, and since $\gamma_i$ is not an edge path it gives a point $p_j$ in the interior of an edge $e_j$ of $lk(v_j,\Psi)$. Since $p_j$ is not a vertex and the length of the edges of the link are $\pi/3$, we must have that $d(v_j', p_j) > \pi$ in the link, and 
    moreover, both vertices of $e_j$ are distance $\geq \pi$ from $v_j'$.
    But now the rotation $\gamma_0'$ gives points $p_j'$ in $lk(v_j,\Psi)$ which are vertices of the edge $e_j$, and hence has distance at least $\pi$ from $v_j'$. This implies that $\gamma'$ is locally geodesic at $v_j$ for $j=1$ and $j=n$. Hence $\gamma'$ is a closed local geodesic in $\Psi$.
    But $\gamma'$ is also an edge path, so by our previous remarks, we must have $\ell(\gamma) = \ell(\gamma') \geq 2\pi$.

    This exhausts all possibilities for $\gamma$, so we see that any local geodesic loop must have length $\geq 2\pi$, implying $\Psi$ is $\mathrm{CAT}(1)$.
\end{proof}

\section{The Moussong metric}
\label{sec:moussong}

In this section, we complete the proof of Theorem \ref{thm:moussong}.
Our key lemma is the following, which we spend the section proving.

\begin{lemma} \label{lem:Gcat1}
    If $G_\Gamma$ is finite with $\Gamma$ connected and not type $A_4(3)$, then $\widehat \Theta_\Gamma$ is $\mathrm{CAT}(1)$ under the Moussong metric.
\end{lemma}

We note that the case where $\Gamma$ is Coxeter is well known, since $\widehat \Theta_\Gamma$ is the Coxeter complex, which, under the Moussong metric, is isometric to a sphere with its usual round metric. It remains to show the Lemma for non-Coxeter $\Gamma$. Our methods for proving the Lemma are case-specific, and as of yet we are unable to treat the $A_4(3)$ Shephard group. We hope to find a unified presentation in the future, hopefully one that includes $A_4(3)$.

We note here that, assuming the lemma, the proof of Theorem \ref{thm:moussong} follows quickly:

\begin{proof}[Proof of Theorem \ref{thm:moussong} (assuming Lemma \ref{lem:Gcat1})]
    Let $\Gamma$ be an extended Coxeter diagram with no subdiagram of the form $A_4(3)$ and with $\mathcal S^f = \mathcal S^{fs}$. By Proposition \ref{prop:linkdecomp}, the link of a vertex $v_T$ for $T \in \mathcal S^f$ in the local development decomposes as $lk(v_T, F_T) * \widehat \Theta_{\Gamma(T)}$. 
    Suppose $\Gamma(T)$ is connected.
    Since $T$ generates a finite Shephard group and $\Gamma(T)$ is not $A_4(3)$, Lemma \ref{lem:Gcat1} implies $\widehat \Theta_{\Gamma(T)}$ is $\mathrm{CAT}(1)$. If $\Gamma(T)$ is the disjoint union $\Gamma_1 \sqcup \ldots \sqcup \Gamma_n$ of connected subdiagrams of $\Gamma$, then by Corollary \ref{cor:ThetaSplits}, we know that $\widehat \Theta_{\Gamma(T)}$ is isometric to the spherical join $\widehat \Theta_{\Gamma_1} * \ldots * \widehat \Theta_{\Gamma_n}$. By assumption, no $\Gamma_i$ is $A_4(3)$, so each $\widehat \Theta_{\Gamma_n}$ is $\mathrm{CAT}(1)$, and hence so is their spherical join $\widehat \Theta_{\Gamma(T)}$.

    Let $K_0 = |(\mathcal S^f_{> \varnothing})'|$. We identify $K_0$ with a subspace of $K$; namely, it is easy to verify that $lk(v_\varnothing, F_\varnothing) \cong K_0$ as simplicial complexes. We put the subspace metric on $K_0$ coming from this identification. It is also straightforward to verify that $lk(v_T, K_0) \cong lk(v_T, F_T)$. Thus by \cite[Lem. 4.4.1]{charney1995k}, $lk(v_T, F_T)$ is $\mathrm{CAT}(1)$. Therefore the orthogonal join of $\widehat \Theta_{\Gamma(T)}$ and $lk(v_T, F_T)$ is $\mathrm{CAT}(1)$. Since all vertices in the local development are a translate of some $v_T$, this means the local developments are nonpositively curved. Since $\mathcal G$ is a simple complex of groups over a simply connected fundamental domain, it follows that $\mathcal G$ has $\mathrm{CAT}(0)$ development $\Theta$. By the definition of $\mathcal G$, $G$ acts properly (all stabilizers are conjugates of the finite parabolics of $G$) and cocompactly (the fundamental domain is the compact space $K$) on $\Theta$, and hence $G$ is $\mathrm{CAT}(0)$.
\end{proof}

\subsection{The 2-generator Shephard groups}

We begin with the 2-generator, or dihedral, Shephard groups $\Gamma = I_2(p,m,q)$. 
Let $\mathcal P$ be the regular complex polygon associated to $\Gamma$. 
By \cite{sunday1975hypergraphs}, the ``girth'' of $\mathcal P$ 
is $m$. In said article, the polygon (which is notated ``$p\{m\}q$'')  is represented as a ``hypergraph'', and ``girth of the polygon'' means the combinatorial girth of this hypergraph, i.e., the minimal number of edges in a nontrivial closed loop. 
Since our complex $\widehat \Theta$ is the incidence graph of this polytope (and hence the corresponding hypergraph), it follows that the (combinatorial) girth of $\widehat \Theta$ is $2m$. Since the length of an edge is $\pi/m$, it follows that any non-trivial loop in $\widehat \Theta$ has length at least $2\pi$, and hence $\widehat \Theta$ is $\mathrm{CAT}(1)$.

\begin{ex}
    Consider $\Gamma = I_2(p,4,2)$ for any $p \geq 2$. 
    The associated complex $\widehat \Theta_\Gamma$ is isomorphic to the barycentric subdivision of $K_{p,p}$. See \cite[\S 4.8]{1975regular} for more details. Depending on which non-starry regular complex polytope for $\Gamma$ we prefer (since it is not self-dual), the vertices of $K_{p,p}$ come from the vertices (or edges) of the complex polytope for $\Gamma$, and the barycenters of edges of $K_{p,p}$ come from the edges (resp.~vertices) of the polytope. 
    Some diagrams for low values of $p$ are given in Figure \ref{fig:I2p}. 
    Black dots are the vertices (or edges) of the polytope for $\Gamma$, and white dots are edges (resp.~vertices) of the polytope. There is an edge between two dots if the respective faces are nested.
\end{ex}

\begin{figure}[ht]
    \begin{tikzpicture}
    \matrix[matrix of nodes,column sep=-10pt,nodes={anchor=center, minimum height=0.1cm, align=flush center}]{
    \begin{tikzpicture}
        \node at (0,0) {$p=3:$};
    \end{tikzpicture}
    &
    \begin{tikzpicture}

        \coordinate (A1) at (90:2);
        \coordinate (A2) at (210:2);
        \coordinate (A3) at (330:2);

        \coordinate (B1) at (90: 0.75);
        \coordinate (B2) at (210:0.75);
        \coordinate (B3) at (330:0.75);

        \draw (A1) -- (B1);
        \draw (A1) -- (B2);
        \draw (A1) -- (B3);

        \draw (A2) -- (B1);
        \draw (A2) -- (B2);
        \draw (A2) -- (B3);

        \draw (A3) -- (B1);
        \draw (A3) -- (B2);
        \draw (A3) -- (B3);

        \coordinate (A1B1) at ($(A1)!0.5!(B1)$);
        \coordinate (A1B2) at ($(A1)!0.5!(B2)$);
        \coordinate (A1B3) at ($(A1)!0.5!(B3)$);
        \coordinate (A2B1) at ($(A2)!0.5!(B1)$);
        \coordinate (A2B2) at ($(A2)!0.5!(B2)$);
        \coordinate (A2B3) at ($(A2)!0.5!(B3)$);
        \coordinate (A3B1) at ($(A3)!0.5!(B1)$);
        \coordinate (A3B2) at ($(A3)!0.5!(B2)$);
        \coordinate (A3B3) at ($(A3)!0.5!(B3)$);

        \filldraw (A1) circle (0.05cm);
        \filldraw (A2) circle (0.05cm);
        \filldraw (A3) circle (0.05cm);
        \filldraw (B1) circle (0.05cm);
        \filldraw (B2) circle (0.05cm);
        \filldraw (B3) circle (0.05cm);

        \draw[fill=white] (A1B1) circle (0.05cm);
        \draw[fill=white] (A1B2) circle (0.05cm);
        \draw[fill=white] (A1B3) circle (0.05cm);
        \draw[fill=white] (A2B1) circle (0.05cm);
        \draw[fill=white] (A2B2) circle (0.05cm);
        \draw[fill=white] (A2B3) circle (0.05cm);
        \draw[fill=white] (A3B1) circle (0.05cm);
        \draw[fill=white] (A3B2) circle (0.05cm);
        \draw[fill=white] (A3B3) circle (0.05cm);
    \end{tikzpicture}
    &
    \hspace{20pt}
    \begin{tikzpicture}
        \node at (0,0) {$p=4:$};
    \end{tikzpicture}
    &
    \begin{tikzpicture}
        \coordinate (A1) at (45: 2);
        \coordinate (A2) at (135:2);
        \coordinate (A3) at (225:2);
        \coordinate (A4) at (315:2);

        \coordinate (B1) at (0  :1.15);
        \coordinate (B2) at (90 :1.15);
        \coordinate (B3) at (180:1.15);
        \coordinate (B4) at (270:1.15);

        \draw (A1) -- (B1);
        \draw (A1) -- (B2);
        \draw (A1) -- (B3);
        \draw (A1) -- (B4);

        \draw (A2) -- (B1);
        \draw (A2) -- (B2);
        \draw (A2) -- (B3);
        \draw (A2) -- (B4);

        \draw (A3) -- (B1);
        \draw (A3) -- (B2);
        \draw (A3) -- (B3);
        \draw (A3) -- (B4);

        \draw (A4) -- (B1);
        \draw (A4) -- (B2);
        \draw (A4) -- (B3);
        \draw (A4) -- (B4);

        \coordinate (A1B1) at ($(A1)!0.5!(B1)$);
        \coordinate (A1B2) at ($(A1)!0.5!(B2)$);
        \coordinate (A1B3) at ($(A1)!0.5!(B3)$);
        \coordinate (A1B4) at ($(A1)!0.5!(B4)$);
        \coordinate (A2B1) at ($(A2)!0.5!(B1)$);
        \coordinate (A2B2) at ($(A2)!0.5!(B2)$);
        \coordinate (A2B3) at ($(A2)!0.5!(B3)$);
        \coordinate (A2B4) at ($(A2)!0.5!(B4)$);
        \coordinate (A3B1) at ($(A3)!0.5!(B1)$);
        \coordinate (A3B2) at ($(A3)!0.5!(B2)$);
        \coordinate (A3B3) at ($(A3)!0.5!(B3)$);
        \coordinate (A3B4) at ($(A3)!0.5!(B4)$);
        \coordinate (A4B1) at ($(A4)!0.5!(B1)$);
        \coordinate (A4B2) at ($(A4)!0.5!(B2)$);
        \coordinate (A4B3) at ($(A4)!0.5!(B3)$);
        \coordinate (A4B4) at ($(A4)!0.5!(B4)$);

        \filldraw (A1) circle (0.05cm);
        \filldraw (A2) circle (0.05cm);
        \filldraw (A3) circle (0.05cm);
        \filldraw (A4) circle (0.05cm);
        \filldraw (B1) circle (0.05cm);
        \filldraw (B2) circle (0.05cm);
        \filldraw (B3) circle (0.05cm);
        \filldraw (B4) circle (0.05cm);

        \draw[fill=white] (A1B1) circle (0.05cm);
        \draw[fill=white] (A1B2) circle (0.05cm);
        \draw[fill=white] (A1B3) circle (0.05cm);
        \draw[fill=white] (A1B4) circle (0.05cm);
        \draw[fill=white] (A2B1) circle (0.05cm);
        \draw[fill=white] (A2B2) circle (0.05cm);
        \draw[fill=white] (A2B3) circle (0.05cm);
        \draw[fill=white] (A2B4) circle (0.05cm);
        \draw[fill=white] (A3B1) circle (0.05cm);
        \draw[fill=white] (A3B2) circle (0.05cm);
        \draw[fill=white] (A3B3) circle (0.05cm);
        \draw[fill=white] (A3B4) circle (0.05cm);
        \draw[fill=white] (A4B1) circle (0.05cm);
        \draw[fill=white] (A4B2) circle (0.05cm);
        \draw[fill=white] (A4B3) circle (0.05cm);
        \draw[fill=white] (A4B4) circle (0.05cm);
    \end{tikzpicture} \\
    };
    \end{tikzpicture}
    \caption{$\widehat \Theta(I_2(p,4,2))$}
    \label{fig:I2p}
\end{figure}
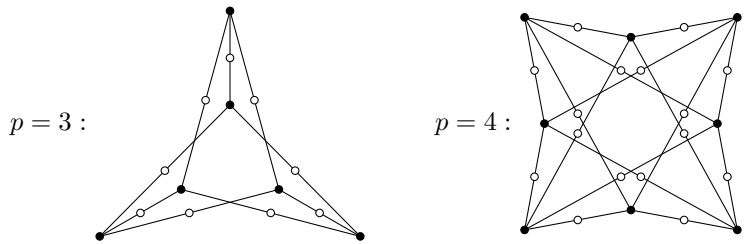

\begin{ex} \label{ex:mobius}
    Let $\Gamma = I_2(3,3,3)$. Then $\widehat \Theta_\Gamma$, shown in Figure \ref{fig:I2333}, is the incidence graph of the \emph{Mobius-Kantor configuration} $8_3$. See \cite[\S 4.8]{1975regular} for more details. The figure given below for $\widehat\Theta_\Gamma$ is from \cite[Fig.~3]{orlik1990milnor}. Here, the black dots represent vertices of this configuration, and the white dots are the lines (or vice versa, since the polytope is self-dual). 
\end{ex}
\begin{figure}[ht]
    \begin{tikzpicture}
    \node[regular polygon,regular polygon sides=8,minimum size=3.8cm, draw] (p) at (0,0) {};

    \coordinate (L1) at ($(p.corner 1)!0.5!(p.corner 2)$);
    \coordinate (L2) at ($(p.corner 2)!0.5!(p.corner 3)$);
    \coordinate (L3) at ($(p.corner 3)!0.5!(p.corner 4)$);
    \coordinate (L4) at ($(p.corner 4)!0.5!(p.corner 5)$);{}
    \coordinate (L5) at ($(p.corner 5)!0.5!(p.corner 6)$);
    \coordinate (L6) at ($(p.corner 6)!0.5!(p.corner 7)$);
    \coordinate (L7) at ($(p.corner 7)!0.5!(p.corner 8)$);
    \coordinate (L8) at ($(p.corner 8)!0.5!(p.corner 1)$);

    \draw (p.corner 1) --
          (p.corner 2) --
          (p.corner 3) --
          (p.corner 4) --
          (p.corner 5) --
          (p.corner 6) --
          (p.corner 7) --
          (p.corner 8) --
          (p.corner 1);

    \draw (p.corner 1) -- (L3);
    \draw (p.corner 2) -- (L4);
    \draw (p.corner 3) -- (L5);
    \draw (p.corner 4) -- (L6);
    \draw (p.corner 5) -- (L7);
    \draw (p.corner 6) -- (L8);
    \draw (p.corner 7) -- (L1);
    \draw (p.corner 8) -- (L2);

    \draw[fill=white] (L1) circle (0.05cm);
    \draw[fill=white] (L2) circle (0.05cm);
    \draw[fill=white] (L3) circle (0.05cm);
    \draw[fill=white] (L4) circle (0.05cm);
    \draw[fill=white] (L5) circle (0.05cm);
    \draw[fill=white] (L6) circle (0.05cm);
    \draw[fill=white] (L7) circle (0.05cm);
    \draw[fill=white] (L8) circle (0.05cm);

    \draw[fill=black] (p.corner 1) circle (0.05cm);
    \draw[fill=black] (p.corner 2) circle (0.05cm);
    \draw[fill=black] (p.corner 3) circle (0.05cm);
    \draw[fill=black] (p.corner 4) circle (0.05cm);
    \draw[fill=black] (p.corner 5) circle (0.05cm);
    \draw[fill=black] (p.corner 6) circle (0.05cm);
    \draw[fill=black] (p.corner 7) circle (0.05cm);
    \draw[fill=black] (p.corner 8) circle (0.05cm);

    \end{tikzpicture}
    \caption{$\widehat \Theta(I_2(3,3,3))$}
    \label{fig:I2333}
\end{figure}
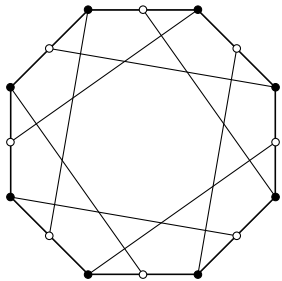

\subsection{The groups \texorpdfstring{$B_n(p,2)$}{B\_n(p)}}

The Shephard groups $G = G_\Gamma$ for $\Gamma = B_n(p,2)$ are unique among the other finite Shephard groups, as they make up the only infinite family of finite non-Coxeter Shephard groups. However, this family becomes easy to deal with because of the following two lemmas.

\begin{lemma} \label{lem:Bnproduct} \emph{\cite[\S 12.2, p.~118]{1975regular}}
    Let $\gamma_n^p$ denote the (unique non-starry) regular complex polytope with symmetry group $G_\Gamma$ for $\Gamma = B_n(p,2)$ (sometimes $\gamma_n^p$ is called a \emph{generalized $n$-cube}). Then, 
    \begin{equation} \label{eq:prodcube}
        \gamma_n^p = \prod^n \gamma_1^p,
    \end{equation}
    where $\prod^n\gamma_1^p$ denotes an $n$-fold direct product of $\gamma_1^p$. 
\end{lemma} 

\begin{lemma} \label{lem:Bnisom}
    The barycentric subdivision of the $n$-fold spherical join of a point has top dimensional cells isometric to the simplex of shape $B_n$. 
\end{lemma}

\begin{proof}
    Let $Q = [0,1]^n$ be the $n$-cube with its standard cellulation. The vertex link $lk(0,Q)$ is an ``all-right simplex'' (i.e., a spherical simplex with all edge lengths equal to $\pi/2$). This is the $n$-fold spherical join of one point (see \cite[Def.~I.5.13]{bridson2013metric} for the definition of spherical join). Let $Q'$ be the barycentric subdivision of $Q$ (see \cite[Def.~I.7.42]{bridson2013metric} for the definition of the barycentric subdivision of metrized cells). Then $lk(v,Q')$ is (simplicially) isometric to $lk(v,Q)'$, the latter of which is the barycentric subdivision of the all-right simplex, and hence the barycentric subdivision of the $n$-fold spherical join of a point. 

    Let $v$ be the vertex of $Q'$ coming from the top dimensional face of $Q$. 
    The symmetry group of $Q$ is the type $B_n$ Coxeter group and thus $lk(v,Q')$ is isometric to the Coxeter complex for the $B_n$ Coxeter group. In particular, the top dimensional cells of $lk(v,Q')$ are simplices of shape $B_n$. It follows easily from the definition of the barycentric subdivision that the top dimensional cells of $lk(0,Q')$ are isometric to the top dimensional cells of $lk(v,Q')$; the result follows.
\end{proof}

From these, we can show

\begin{prop}
    For all $n$ and $p$, $\widehat \Theta(B_n(p,2))$ is isometric to the barycentric subdivision of the $n$-fold spherical join of a set of $p$ points, and hence is $\mathrm{CAT}(1)$.
\end{prop}

\begin{proof}
    First, we recall basic constructions relating to posets. The (upper) cone $c \mathcal P$ on a poset $\mathcal P$ is the poset obtained by adding an element $1_{\mathcal P}$ (called the \emph{cone point}) to $\mathcal P$ which we declare maximal in the order on $\mathcal P$. The derived complex $\mathcal P'$ is the set of all linearly ordered subsets (or ``chains'') of $\mathcal P$, ordered by inclusion. The derived complex $\mathcal P'$ is always an abstract simplicial complex, and we denote its geometric realization by $|\mathcal P'|$. 
    The product $\mathcal P \times \mathcal Q$ is the poset whose underlying set is usual set-theoretic product $\mathcal P \times \mathcal Q$ with order $(p,q) \leq (p',q')$ if and only if $p \leq p'$ and $q \leq q'$. 
    The join $\mathcal P * \mathcal Q$ is the poset $(c\mathcal P \times c \mathcal Q) \setminus \{(1_{\mathcal P}, 1_{\mathcal Q})\}$.
    We note that there is an order-preserving isomorphism $c(\mathcal P * \mathcal Q) \cong c \mathcal P \times c \mathcal Q$ which maps the cone point of $c(\mathcal P * \mathcal Q)$ to $(1_{\mathcal P}, 1_{\mathcal Q})$.

    We now turn to the setting of $B_n(p,2)$. Let $(\gamma_n^p)_{prop}$ denote the poset of proper faces of $\gamma_n^p$ (i.e., all faces except $\varnothing$ and $\mathbb{C}^n$), and let $(\gamma_n^p)_{>\varnothing}$ denote the set of non-empty faces of $\gamma_n^p$ (but including $\mathbb{C}^n$). 
    The poset $(\gamma_n^p)_{>\varnothing}$ is the cone on $(\gamma_n^p)_{prop}$ with cone point $\mathbb{C}^n$.
    By the comment after Definition \ref{def:prodpoly}, there is a poset isomorphism
    \begin{align*}
        \prod^n (\gamma_1^p)_{>\varnothing} \cong \left( \prod^n \gamma_1^p  \right)_{>\varnothing},
    \end{align*}
    where the product on the left hand side is the poset product, and that on the right is the polytope product. Thus by Lemma \ref{lem:Bnproduct} we have 
    \begin{align*}
        \prod^n  c(\gamma_1^p)_{prop} = \prod^n (\gamma_1^p)_{>\varnothing} \cong (\gamma_n^p)_{>\varnothing} = c (\gamma_n^p)_{prop}.
    \end{align*}
    And by our previous comments about products and joins of posets, we see that this gives the poset isomorphism
    \begin{align*}
        c\left(\bigast^n (\gamma_1^p)_{prop}\right) \cong  c (\gamma_n^p)_{prop},
    \end{align*}
    where $\bigast^n$ is the $n$-fold spherical join.
    In particular, the cone points are preserved in this map, so
    \begin{align*}
        \bigast^n (\gamma_1^p)_{prop} \cong   (\gamma_n^p)_{prop}.
    \end{align*}
    Thus by taking geometric realizations, we have a simplicial homeomorphism
    \begin{align*}
        \big|\big(\bigast^n (\gamma_1^p)_{prop}\big)'\big| \cong  | (\gamma_n^p)_{prop}'|.
    \end{align*}
    By our definitions, $\widehat \Theta \cong | (\gamma_n^p)_{prop}'|$. Moreover, the diagram $B_1(p,2)$ is a single vertex labeled $p$; its regular complex polytope $\gamma_1^p$ is a copy of $\mathbb{C}$ with $p$ distingished points as vertices, so the poset $(\gamma_1^p)_{prop}$ is a set of $p$ points, none of which are comparable.
    Thus $\bigast^n (\gamma_1^p)_{prop}$ is the poset of cells of the $n$-fold spherical join of a set of $p$ points, which we will call $\Delta$, and $\big(\bigast^n (\gamma_1^p)_{prop}\big)'$ is the poset of cells of the barycentric subdivision $\Delta'$ of $\Delta$. In other words, there is a simplicial homeomorphism
    \begin{align*}
        \Delta' \cong \widehat \Theta(B_n(p,2)).
    \end{align*}
    The Moussong metric declares the top dimensional simplices of $\widehat \Theta(B_n(p,2))$ to be simplices of type $B_n$. Lemma \ref{lem:Bnisom} implies that the top dimensional cells of $\Delta'$ are also simplices of type $B_n$; hence, this map is a simplicial isometry. Since $\Delta$ is $\mathrm{CAT}(1)$ (it is a spherical join of $\mathrm{CAT}(1)$ spaces), then $\Delta'$ is $\mathrm{CAT}(1)$ \cite[Lem.~I.7.48]{bridson2013metric}, and thus so is $\widehat \Theta(B_n(p,2))$.
\end{proof}

\subsection{The group \texorpdfstring{$A_3(3)$}{A\_3(3)}}

We now consider $\widehat \Theta = \widehat \Theta(A_3(3))$. 
Throughout this section, let $\mathcal H$ denote the regular complex polytope associated to $A_3(3)$. Our arguments rely on the combinatorial structure of this polytope---see Appendix \ref{sec:hessiancomb} for more details. 

Here we utilize Charney's combinatorial $\mathrm{CAT}(1)$ criteria (CCCC) that we introduced in Section \ref{sec:cat1crit}.
Note that $\widehat \Theta$ carries a natural marked $A_3$ simplicial complex structure: we call a vertex of $\widehat \Theta$ type $\hat a$ if it is a vertex of $\mathcal H$, type $\hat b$ if it is an edge of $\mathcal H$, and type $\hat c$ if it is a face of $\mathcal H$. 
We also note that the canonical metric for marked $A_3$ simplicial complexes agrees with the Moussong metric on $\widehat \Theta$; both place an angle of $\pi/3$ at the vertices of type $\hat a$ and $\hat c$, and an angle of $\pi/2$ at the vertices of type $\hat b$.

\begin{prop}
    Under the above marking, $\widehat \Theta(A_3(3))$ satisfies \tempname, and is therefore $\mathrm{CAT}(1)$ under its canonical metric (and hence the Moussong metric).
\end{prop}

\begin{proof}

The links of vertices of type $\hat a$ and $\hat c$ are isomorphic to $\widehat \Theta(I_2(3,3,3))$ which has girth $6$, and the link of a vertex of type $\hat b$ is isomorphic to $\widehat \Theta(I_2(3,2,3))$, which by Corollary \ref{cor:ThetaSplits} is isomorphic to $K_{3,3}$. 

It remains to show that $\widehat \Theta$ satisfies 
condition (\ref{item:fillloop}) of the criteria.

\underline{Path (i)}: Translating to $\mathcal H$, this setup is equivalent to taking two faces of $\mathcal H$ (corresponding to the vertices of type $\hat c$) which intersect in two distinct vertices of $\mathcal H$ (the vertices of type $\hat a$). By Proposition \ref{prop:faceintersectsimple}, these faces must intersect along an edge containing these two vertices. This edge in the polytope corresponds to a vertex of type $\hat b$ which is joined to each vertex of $\gamma$, and each of the triangles formed are filled as in Figure \ref{fig:fillededge}(i).

\underline{Path (ii)}: 
Path (ii.a) says we have three edges in the complex polytope (vertices of type $\hat b$) which pairwise intersect at three distinct vertices of the polytope (vertices of type $\hat a$). 
By Proposition \ref{prop:threeedgesintersect}, these edges must be contained in some common face (a vertex of type $\hat c$), giving Figure \ref{fig:fillededge}(ii.a). The (ii.b) case follows from passing to the dual polytope.

Thus we conclude that $\widehat \Theta(A_3(3))$ satisfies CCCC, and is therefore $\mathrm{CAT}(1)$.
\end{proof}

\subsection{The group \texorpdfstring{$B_3(2,3)$}{B\_3(3,3)}}

By \cite[Sec. 12.4]{1975regular}, we know $G_{B_3(2,3)} \cong (\mathbb{Z}/{2}\mathbb{Z}) \times G_{A_3(3)}$. This manifests on the level of complexes; $\widehat \Theta(B_3(2,3))$ is a subdivision of $\widehat \Theta(A_3(3))$. Specifically, doubling the fundamental domain for $\widehat \Theta(B_3(2,3))$ along the face corresponding to the order-2 generator is isometric to the fundamental domain for $\widehat \Theta(A_3(3))$. 
This is illustrated in Figure \ref{fig:B3subdivA3}, where the fundamental domain for $B_3(2,3)$ is outlined in bold, with the corresponding stabilizers labeling the edges. The vertices of $\Delta_{A_3(3)}$ are labeled according to the convention used in the previous section. 
Since this subdivision respects the metric on $\widehat \Theta(A_3(3))$, it follows that $\widehat \Theta(B_3(2,3))$ is $\mathrm{CAT}(1)$.

\begin{figure}[!ht]
    \begin{tikzpicture}
        \coordinate (A) at (0,0);
        \coordinate (B) at (-3,3);
        \coordinate (C) at (3,3);
        \coordinate (D) at (0,4.2);

        \draw[fill=gray!25] (A) -- (D) to[out=0,in=135]  (C) -- (A);

        \draw[line width=0.05cm,fill=gray!25] (A) -- (B) to[out=45,in=180] (D) -- (A);

        \filldraw (A) circle (0.05cm);
        \filldraw (B) circle (0.05cm);
        \filldraw (C) circle (0.05cm);
        \filldraw (D) circle (0.05cm);

        \node at (-1.1,2.5) {$\Delta_{B_3(2,3)}$};

        \node[below=0.1cm] at (A) {$\hat b$};
        \node[left=0.1cm] at (B) {${\hat a}$};
        \node[right=0.1cm] at (C) {$\hat c$};

        \node at (0.25,0.65) {$\frac{\pi}{4}$};
        \node at (-0.25,0.65) {$\frac{\pi}{4}$};
        
        \node[right=0.15cm] at (B) {$\frac{\pi}{3}$};
        \node[left=0.15cm] at (C) {$\frac{\pi}{3}$};

        \node[below left] at (D) {$\frac{\pi}{2}$};
        \node[below right] at (D) {$\frac{\pi}{2}$};

        \node[right] at (0,2.2) {$\mathbb{Z}/2\mathbb{Z}$};
        \node[below left] at (-1.5,1.5) {$\mathbb{Z}/3\mathbb{Z}$};
        \node[above left] at (-1.3,3.9) {$\mathbb{Z}/3\mathbb{Z}$};
    \end{tikzpicture}
    \caption{$\Delta_{B_3(2,3)}$ as a subspace of $\Delta_{A_3(3)}$}
    \label{fig:B3subdivA3}
\end{figure}
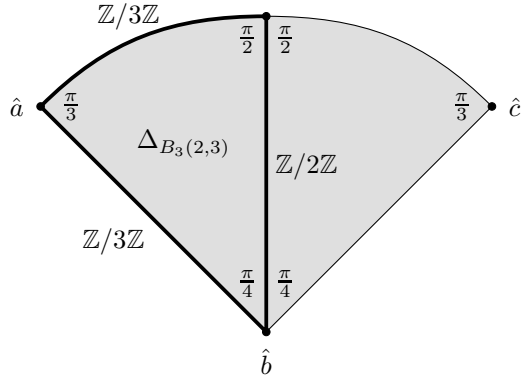

This exhausts the list of finite Shephard groups with connected diagram (excluding $A_4(3)$), thus completing the proof of Lemma \ref{lem:Gcat1}, and by extension, Theorem \ref{thm:moussong}.

\begin{appendix}

\section{Combinatorics of the Hessian polyhedron}
\label{sec:hessiancomb}

In this section we provide explanations for many technical lemmas regarding the $A_3(3)$ Shephard group and lay out explicit data utilized in the computations  therein. Throughout, let $\mathcal H$ denote the standard Hessian polyhedron, that is, the non-starry regular complex polytope with diagram $A_3(3)$. We begin by recalling basic facts about the facets of $\mathcal H$. The following can be found in \cite[Ch.~12.3]{1975regular}.

The 27 vertices of $\mathcal H$ are given by 
\[
    (0, \omega^i, -\omega^j) \qquad (-\omega^j, 0, \omega^i) \qquad (\omega^i, -\omega^j, 0)
\]
where $\omega = \exp(2\pi \sqrt{-1}/3)$ and $i,j = 1,2,3$. As in Coxeter, we find it convenient to use the shorthand
\[
    0ij \qquad j0i \qquad ij0
\]
for $i,j =1,2,3$ to represent the respective vertices. (For example, $013$ is the vertex $(0, \omega^1, -\omega^3)$.) The generators of the $A_3(3)$ Shephard group have a nice description as permutations of these vertices, which is listed in Table \ref{tab:permuteA3}.

\begin{table}[ht] \small
\begin{align*}
    a &= (101\ 201\ 301)(102\ 202\ 302)(103\ 203\ 303)(110\ 210\ 310)(120\ 220\ 320)(130\ 230\ 330) \\[0.45em]
    b &=  (012\ 230\ 103) (013\ 102\ 320) (021\ 203\ 130)(023\ 310\ 201)(031\ 120\ 302)(032\ 301\ 210) \\[0.45em]
    c &= (011\ 012\ 013)(021\ 022\ 023)(031\ 032\ 033)(101\ 102\ 103)(201\ 202\ 203)(301\ 302\ 303).
\end{align*}
\caption{Permutation representation of $A_3(3)$}
\label{tab:permuteA3}
\end{table}

There are 72 edges of $\mathcal H$. Two vertices lay on a common edge if and only if their symbols agree in an even number of positions, meaning they agree in two positions or no position. For example, $013$ and $023$ share an edge, while $013$ and $022$ do not. Thus it is straightforward to verify if two given points share an edge in $\mathcal H$ despite the (somewhat) large number of edges. In addition, since ``edges'' of a complex polytope are (affine) complex lines, and two points determine a unique line, two vertices share \emph{at most} one edge. 

We next describe the (2-)faces. Since $\mathcal H$ is self-dual, it has 27 faces. One of the faces has vertices
\[
    012, 021, 103, 203, 303, 130, 230, 330,
\]
which is the orbit of $012$ under the subgroup generated by $a$ and $b$.
The other faces are the translates of this face by elements of the symmetry group. We provide the vertices of these faces here for posterity (including the face mentioned above) in Table \ref{tab:hessianFaces}.
\begin{table}[ht]
\begin{align*}
 1:    \ &  0 1 2 \ 0 2 1 \ 1 0 3 \ 2 0 3 \ 3 0 3 \ 1 3 0 \ 2 3 0 \ 3 3 0 &
 2:    \ &  3 0 2 \ 0 2 2 \ 0 2 3 \ 1 1 0 \ 3 0 1 \ 3 0 3 \ 2 3 0 \ 0 2 1 \\
 3:    \ &  2 0 2 \ 0 2 2 \ 3 1 0 \ 2 0 1 \ 0 2 3 \ 1 3 0 \ 0 2 1 \ 2 0 3 &
 4:    \ &  3 2 0 \ 0 2 2 \ 1 0 1 \ 3 1 0 \ 0 3 2 \ 2 0 3 \ 0 1 2 \ 3 3 0 \\
 5:    \ &  1 0 2 \ 0 2 2 \ 1 0 1 \ 0 2 3 \ 2 1 0 \ 0 2 1 \ 1 0 3 \ 3 3 0 &
 6:    \ &  0 1 1 \ 0 2 3 \ 1 0 2 \ 2 0 2 \ 3 0 2 \ 1 3 0 \ 2 3 0 \ 3 3 0 \\
 7:    \ &  2 2 0 \ 0 2 2 \ 2 1 0 \ 0 3 2 \ 3 0 1 \ 0 1 2 \ 2 3 0 \ 1 0 3 &
 8:    \ &  1 2 0 \ 0 2 2 \ 0 3 2 \ 2 0 1 \ 1 1 0 \ 1 3 0 \ 3 0 3 \ 0 1 2 \\
 9:    \ &  0 3 1 \ 0 2 2 \ 3 1 0 \ 1 1 0 \ 2 1 0 \ 3 0 3 \ 1 0 3 \ 2 0 3 &
 10:   \ &  0 1 1 \ 3 1 0 \ 3 2 0 \ 2 0 2 \ 0 3 1 \ 0 2 1 \ 1 0 3 \ 3 3 0 \\
 11:   \ &  3 2 0 \ 0 2 3 \ 1 0 2 \ 3 1 0 \ 0 3 3 \ 2 0 1 \ 0 1 3 \ 3 3 0 &
 12:   \ &  2 2 0 \ 0 2 3 \ 2 1 0 \ 0 3 3 \ 3 0 2 \ 0 1 3 \ 2 3 0 \ 1 0 1 \\
 13:   \ &  1 2 0 \ 0 2 3 \ 0 3 3 \ 2 0 2 \ 1 1 0 \ 1 3 0 \ 3 0 1 \ 0 1 3 &
 14:   \ &  0 3 2 \ 0 2 3 \ 3 1 0 \ 1 1 0 \ 2 1 0 \ 3 0 1 \ 1 0 1 \ 2 0 1 \\
 15:   \ &  2 2 0 \ 0 2 1 \ 2 1 0 \ 0 3 1 \ 3 0 3 \ 0 1 1 \ 2 3 0 \ 1 0 2 &
 16:   \ &  1 2 0 \ 0 2 1 \ 0 3 1 \ 2 0 3 \ 1 1 0 \ 1 3 0 \ 3 0 2 \ 0 1 1 \\
 17:   \ &  0 3 3 \ 0 2 1 \ 3 1 0 \ 1 1 0 \ 2 1 0 \ 3 0 2 \ 1 0 2 \ 2 0 2 &
 18:   \ &  2 2 0 \ 3 1 0 \ 0 3 2 \ 0 3 3 \ 0 3 1 \ 1 0 2 \ 1 0 3 \ 1 0 1 \\
 19:   \ &  2 2 0 \ 2 0 3 \ 0 3 2 \ 1 2 0 \ 3 0 3 \ 0 1 1 \ 1 0 3 \ 3 2 0 &
 20:   \ &  0 3 3 \ 2 0 3 \ 2 0 1 \ 1 1 0 \ 0 3 2 \ 0 3 1 \ 3 2 0 \ 2 0 2 \\
 21:   \ &  2 2 0 \ 2 0 1 \ 3 0 1 \ 0 3 3 \ 1 2 0 \ 3 2 0 \ 0 1 2 \ 1 0 1 &
 22:   \ &  2 2 0 \ 1 3 0 \ 3 0 1 \ 3 0 2 \ 3 0 3 \ 0 1 1 \ 0 1 2 \ 0 1 3 \\
 23:   \ &  0 3 3 \ 3 0 3 \ 3 0 1 \ 2 1 0 \ 0 3 2 \ 0 3 1 \ 1 2 0 \ 3 0 2 &
 24:   \ &  3 2 0 \ 2 3 0 \ 1 0 1 \ 1 0 2 \ 1 0 3 \ 0 1 1 \ 0 1 2 \ 0 1 3 \\
 25:   \ &  2 2 0 \ 2 0 2 \ 3 0 2 \ 0 3 1 \ 1 2 0 \ 3 2 0 \ 0 1 3 \ 1 0 2 &
 26:   \ &  1 2 0 \ 3 3 0 \ 2 0 1 \ 2 0 2 \ 2 0 3 \ 0 1 1 \ 0 1 2 \ 0 1 3 \\
 27:   \ &  0 1 3 \ 0 2 2 \ 1 0 1 \ 2 0 1 \ 3 0 1 \ 1 3 0 \ 2 3 0 \ 3 3 0 &
\end{align*}
\caption{The 27 faces of $\mathcal H$}
\label{tab:hessianFaces}
\end{table}
We note that the list of vertices is sufficient to determine the entire face; an edge is included in a face if and only if each of its vertices are. 
That is to say that if three vertices in a face span an edge of $\mathcal H$, then this edge is part of said face.

We conclude with two further observations which can be derived from examining Table \ref{tab:hessianFaces}. First,

\begin{prop} \label{prop:threeedgesintersect}
    If $E_1$, $E_2$, and $E_3$ are distinct edges of $\mathcal H$ which pairwise intersect non-trivially but have trivial total intersection, then there is a face $F$ of $\mathcal H$ containing $E_1$, $E_2$, and $E_3$.
\end{prop}

This essentially says there are no ``empty triangles'' in $\mathcal H$. The following says that the intersection of faces is well behaved in a certain way.

\begin{prop} \label{prop:faceintersectsimple}
    If $F_1$ and $F_2$ are distinct faces of $\mathcal H$, then $F_1 \cap F_2$ is either empty, a single vertex, or a single edge.
\end{prop}

Again, both of these may be proven ``by hand'' (or more conveniently, by computer) using Table \ref{tab:hessianFaces} and the previously stated characterization of the edges of $\mathcal H$.

\end{appendix}

\bibliographystyle{alpha}
\bibliography{cat0shep}

\end{document}